\documentclass{birkjour}

\usepackage{amsmath,amsthm,amssymb,mathtools}  
\usepackage{dsfont}				

\usepackage{enumitem}
\usepackage{graphicx}			
\usepackage{subcaption}			
\usepackage[font=small]{caption}	

\usepackage[latin1]{inputenc}   

\usepackage[noadjust]{cite}		

\usepackage{verbatim}		

\usepackage[
disable,					
textsize=scriptsize,textwidth=4.2cm]{todonotes}	
\newcommand{\SELF}[1]{\todo[color=green!40]{#1}} 
\newcommand{\OMIT}[1]{\todo[color=gray!30]{#1}}  
\newcommand{\CITE}[1]{\todo[color=cyan!30]{#1}}  
\newcommand{\FUTURE}[1]{\todo[color=purple!30]{#1}}  

\usepackage[hyperindex,breaklinks]{hyperref}	
\hypersetup{colorlinks=true, linkcolor=blue, citecolor=blue}

\usepackage{orcidlink}			
\usepackage[nameinlink]{cleveref} 	

\usepackage{tikz-cd}			
\usepackage{bm}					
\usepackage{booktabs}			

\newtheorem{theorem}{Theorem}[section]
\newtheorem{proposition}[theorem]{Proposition}
\newtheorem{corollary}[theorem]{Corollary}
\newtheorem{lemma}[theorem]{Lemma}

\theoremstyle{definition}
\newtheorem{definition}[theorem]{Definition} 
\newtheorem{example}[theorem]{Example} 
\newtheorem*{notation}{Notation} 
\newtheorem{remark}[theorem]{Remark}

\def\R			{\mathds{R}}
\def\C			{\mathds{C}}
\def\F			{\mathds{F}}
\def\Id			{\mathds{1}}

\def\OO			{\mathcal{O}}
\def\VV			{\mathcal{V}}
\def\AA			{\mathcal{A}}
\def\PP			{\mathcal{P}}

\def\wrt					{w.r.t.\ }

\def\resp					{resp.\ }

\def\im	 					{\mathrm{i}}				

\newcommand\inner[1] 		{\langle #1 \rangle}		

\def\rwedge			{\mathbin{\vartriangle}}
\def\cwedge			{\wedge}

\DeclareMathOperator{\Span}{span}

\DeclareMathOperator{\Real}{Re}
\DeclareMathOperator{\Imag}{Im}
\DeclareMathOperator{\diag}{diag}
\DeclareMathOperator{\sgn}{sgn}

\begin{document}

\title[Hermitian Geometry of Complex Multivectors and Determinants]{Hermitian Geometry of Complex Multivectors, Determinants and Orientations}

\author{Andr\'e L. G. Mandolesi }
\address{Instituto de Matem\'atica e Estat\'istica \\ Universidade Federal da Bahia \\ Av. Milton Santos s/n \\ 40170-110 Salvador - BA \\ Brazil. \\ ORCID 0000-0002-5329-7034}
\email{andre.mandolesi@ufba.br}

\date{\today}

\begin{abstract}
	Two geometric interpretations for complex multivectors and determinants are presented: a little known one in terms of square roots of volumes, and a new one using fractions of volumes. The fraction is determined by a new holomorphy index, which measures the lack of holomorphy of real subspaces of $\C^n$ via generalized Kähler angles or a disjointness angle, and relates real and complex exterior products.	The interpretations are completed with a natural but uncommon concept of complex orientation, related to elementary complex transformations. We also propose graphical representations for complex blades (decomposable multivectors) as fractions of parallelotopes, and discuss how Clifford algebras relate (or not) to Hermitian geometry.
\end{abstract}



\maketitle

\section{Introduction}

Real and complex Grassmann exterior algebras \cite{Bourbaki1989,Greub1978,Shaw1983} are widely used in geometry, analysis, representation theory, physics, etc., specially via differential forms \cite{Frankel2011,Griffiths1994,Huybrechts2004,Rosen2019}. 
Forms or multi-covectors are usually interpreted as alternating multilinear functions, algebraic properties link them to determinants, and Jacobians make them well suited for calculus on manifolds. 

In the real case, determinants and blades (decomposable multivectors) also have useful geometric interpretations in terms of oriented volumes or parallelotopes.
And real Clifford geometric algebras provide even more powerful ways to explore the geometry of real multivectors \cite{Corrochano2001,Hestenes1984,Rosen2019}.

In the complex case, the relation with volumes and orientations is less obvious.
This can be a problem even in real settings:
e.g., complex blades can appear in the invariant decomposition of a real bivector \cite{Roelfs2023}.
And while complex Clifford algebras have important applications \cite{Hrdina2022,Sabadini2014,Stoica2018}, they lack a clear connection to the geometry of general Hermitian spaces (complex spaces with Hermitian product, but no real structure providing complex conjugation), which are important in quantum mechanics, gauge theory, etc.

In this article, we describe two ways in which complex determinants and blades are related to volumes.
In the first one, the absolute value of a complex determinant, or the norm of a complex blade, is the square root of a certain volume. 
It follows by interpreting known formulas linking real and complex determinants (via field norms \cite{Bourbaki1989}), often used to explain the natural orientation of complex spaces \cite{Greub1978,Griffiths1994}.
	\CITE{Greub1978 p.188, Griffiths1994 p.18}
This volumetric interpretation has drawn little attention:
after a literature search and inquiries to experts, we only found it in internet forums \cite{Lahtonen2012,Malyshev2016}.
But it has proven surprisingly useful, leading to complex volumetric Py\-thag\-o\-re\-an theorems \cite{Mandolesi_Pythagorean} with fundamental implications for quantum theory \cite{Mandolesi_Born},
and providing geometric interpretations for an asymmetric angle between complex subspaces \cite{Mandolesi_Grassmann,Mandolesi_Products}, for the contraction (interior product) of complex blades \cite{Mandolesi_Contractions},
and for Fubini-Study and Cauchy-Binet metrics in complex Grassmannians \cite{Mandolesi_TotalGrassmannian}.

The second way is new, and uses fractions of volumes determined by a holomorphy index. 
This index measures how a real subspace of $\C^n$ fails to be holomorphic (i.e., the underlying real space of a complex subspace) in terms of how it is displaced by the action of the complex structure.
The displacement is described using a generalization, via principal angles \cite{Bjorck1973,Galantai2006}, of the Kähler or holomorphy angles of real planes \cite{Goldman1999,Rosenfeld1997,Scharnhorst2001} to high dimensional subspaces,  
similar to the multiple Kähler angle of \cite{Tasaki2001}.
These angles combine into a disjointness angle derived from the asymmetric one, whose relation to exterior products \cite{Mandolesi_Grassmann,Mandolesi_Products}
implies that the holomorphy index equals the ratio of the norms of complex and real exterior products of the same vectors,
thus establishing a link between real and complex blades.

These volumetric interpretations are complemented by a new concept of complex orientation, which is natural but unorthodox, since complex spaces are usually seen as naturally oriented \cite{Greub1978,Griffiths1994}.
This natural orientation, however, is actually a real orientation of the underlying real space, and does not reflect complex geometry as well as truly complex orientations.
On the other hand, while real orientations are discrete (two), complex ones form a continuum (a circle), as in each case they can be identified with the unit blades in $\bigwedge^n \R^n \cong \R$ or $\bigwedge^n \C^n \cong \C$.
Although this means they will lack some common uses of the real ones, complex orientations have been related to the contraction of complex multivectors, and used to obtain a simpler Hodge-like star operator for complex spaces \cite{Mandolesi_Contractions}.
Here we analyze them in more detail, giving alternative definitions, relating them to elementary complex transformations, and to the argument of complex determinants.

Our results suggest representing complex blades as fractions of parallelotopes,
with the advantage of indicating the holomorphy of the real subspace spanned by the vectors in a blade decomposition.
However, unless the decomposition is chosen adequately, graphical sums of blades become more complicated than usual, as the holomorphy must be taken into account.
We also propose representations for complex orientations and vectors.

Complex Grassmann algebras are well suited for use with Hermitian geometry,
and so are contractions%
\footnote{Asymmetric versions of Hestenes inner product, with better properties.}
and regressive products \cite{Mandolesi_Contractions,Mandolesi_Contractions2}.
For Clifford algebras this is less clear.
They have been used in complex settings in the form of Hermitian Clifford Analysis \cite{Brackx2008,Sabadini2014},
Unitary Geometric Algebra \cite{Sobczyk1993,Sobczyk2012unitary},
and other ways \cite{Hrdina2022,Marchuk2008,Sobczyk2001universal,Stoica2018}.
	\CITE{Sobczyk2013 does the same; Sobczyk1993 uses a bivector as $\im$; Sobczyk2001universal uses complex algebra}
But, as we discuss, these methods introduce concepts like complex conjugation, or distinctly real and imaginary vectors, which are extraneous to general Hermitian spaces.
This breaks the isotropy of the spaces, an important feature for some applications, e.g. in quantum mechanics and gauge theory.
Worse yet, the complex Clifford product seems incompatible with the Hermitian one and isotropy.
Geometric algebraists might suggest using real Clifford algebras in the underlying real space, but as we show this can be inefficient and cumbersome.
In light of this, we also discuss Hestenes' rejection of complex scalars in his formulation of Clifford Geometric Algebra \cite{Hestenes1986unified,Hestenes1984}.

\Cref{sc:Determinants} presents the first volumetric interpretation for determinants.
The cases are linked by simple algebraic formulas which give the result when properly interpreted, 
but we also provide more geometric proofs, to give different views of the relation between determinants and volumes.
We also show why some real proofs fail in the complex case.
\Cref{sc:Complex orientation} defines and interprets complex orientations, relates them to elementary complex transformations, and interprets the argument of determinants. 
\Cref{sc:Complex blades} transfers the first volumetric interpretation of determinants to blade norms, identifies orientations as unit blades, and obtains a first relation between real and complex blades.
It also shows how to obtain simpler proofs for real and complex volumetric Pythagorean theorems \cite{Mandolesi_Pythagorean}.
\Cref{sc:Holomorphy} reviews the concept of holomorphy for real subspaces of $\C^n$, and defines generalized Kähler angles and the holomorphy index, which linked to the disjointness angle gives a simpler relation between real and complex blades.
\Cref{sc:2nd interpretation} gives the second volumetric interpretation of complex blades and determinants, 
and links the determinant of linear transformations of $\C^n$ to  holomorphy changes and dilation of volumes in $n$-dimensional real subspaces.
\Cref{sc:Visual representations} proposes graphical representations of complex vectors, orientations and multivectors.
\Cref{sc:Complex numbers and Geometric Algebra} discusses the difficulties in applying Clifford algebras to Hermitian geometry. 
\Cref{sc:Conclusion} summarizes our results and discusses possible developments.
%
\Cref{Angles between subspaces} reviews different concepts of angle between subspaces, defines and interprets a disjointness angle, and establishes some properties, in particular its relation to exterior products.

Throughout the article, the well known real case is presented alongside the complex one, for comparison.
Also, several examples are continuously developed, incorporating the results of each new section.

\section{Determinants -- 1st interpretation}\label{sc:Determinants}

Let $\F= \R$ or $\C$, and $\inner{\cdot,\cdot}$ be the canonical inner/Hermitian%
	\footnote{Conjugate-linear in the first entry.}
product in $\F^n$.
When necessary, we write $\Span_\R$, $\dim_\C$, $\R$-linear, etc.\ to indicate the field.
A $p$-dimensional subspace is a \emph{$p$-subspace} (a \emph{line} if $p=1$, a \emph{plane} if $p=2$).
The line spanned by $v\neq 0$ is denoted by $\F v$.

Each complex subspace $V \subset \C^n$ determines an \emph{underlying real subspace} $V_\R \subset \R^{2n}$ via an identification of $v=(x_1+\im y_1,\ldots,x_n+\im y_n) \in \C^n$ with $v=(x_1, y_1,\ldots,x_n, y_n) \in \R^{2n}$.
In $(\C^n)_\R$, the inner product is $\inner{\cdot,\cdot}_\R = \Real \inner{\cdot,\cdot}$,
and $\im$ becomes the \emph{complex structure}
\cite{Goldman1999,Huybrechts2004},
	\CITE{Goldman1999 p.29, Huybrechts2004 p.25}
an operator of rotation by $\frac\pi2$ given by $\im v=(-y_1,x_1,\ldots,-y_n,x_n)$.
For $u,v \in \C^n$ we have $\inner{\im u,\im v}_\R = \inner{u,v}_\R$ and
		\OMIT{$\inner{u,\im v}_\R = \Real(\im \inner{u,v}) = -\Imag \inner{u,v}$}
\begin{equation}\label{eq:inner}
	\inner{u,v} = \inner{u,v}_\R + \im \inner{\im u, v}_\R.
\end{equation}
If $\inner{u,v}_\R = 0$ then $u$ and $v$ are \emph{$\R$-orthogonal} ($u\!\perp_\R\!v$).
They are \emph{$\C$-orthogonal} if $\inner{u,v} = 0$, 
which by \eqref{eq:inner} means $v$ is $\R$-orthogonal to $\C u \cong \Span_\R\{u,\im u\}$.

In the real case, let $V_\R=V$ and $\inner{\cdot,\cdot}_\R = \inner{\cdot,\cdot}$.

The \emph{Euclidean angle}
$\theta_{u,v} = \cos^{-1} \frac{\inner{u,v}_\R}{\|u\|\|v\|}$
of $u,v \in \F^n$ is the angle in $(\F^n)_\R$.
The \emph{Hermitian angle} \cite{Scharnhorst2001}
$\gamma_{u,v} = \cos^{-1} \frac{|\inner{u,v}|}{\|u\|\|v\|}$
is the angle between $\F u$ and  $\F v$.

\begin{definition}\label{df:main}
	Let $v_1,\ldots,v_p \in \F^n$.
	\begin{enumerate}[label = (\roman*)]
		\item $M(v_1,\ldots,v_p)$ is the $n \times p$ matrix with these vectors as columns.
		In the complex case, we also define $M_\R(v_1,\ldots,v_p)$ as the $2n \times p$ real matrix whose columns are these vectors considered in $(\C^n)_\R$.

		\item $G(v_1,\ldots,v_p) = \left(\inner{v_i,v_j}\right)_{p \times p}$ is their \emph{Gram matrix}.
		In the complex case,
		$G_\R(v_1,\ldots,v_p) = \left(\inner{v_i,v_j}_\R\right)_{p \times p}$ is their \emph{Gram matrix in $(\C^n)_\R$}.
		
		\item $\PP(v_1,\ldots,v_p) = \{\sum_{j=1}^p t_j v_j: 0\leq t_i \leq 1\}$ is the \emph{parallelotope} they span.
		
		\item $\VV(v_1,\ldots,v_p) = \VV_p(v_1,\ldots,v_p)$ is the \emph{$p$-volume} of $\PP(v_1,\ldots,v_p)$, defined by $\VV(v_1) = \|v_1\|$ and $\VV(v_1,\ldots,v_k) = \VV(v_1,\ldots,v_{k-1}) \cdot \|u_k\|$ for $k \leq p$, where $u_k = v_k - \sum_{j=1}^{k-1} c_j v_j$ for $c_j \in \R$ such that $\inner{v_j,u_k}_\R = 0$ $\forall\, j<k$. \label{it:volume}
	\end{enumerate}
\end{definition}

Note that $\VV_p$ is the usual $p$-dimensional volume in $(\F^n)_\R$, since each $u_k$ is the component of $v_k$  $\R$-orthogonal to $\Span_\R\{v_1,\ldots,v_{k-1}\}$ and so $\|u_k\|$ is the height of $\PP(v_1,\ldots,v_k)$ \wrt its base $\PP(v_1,\ldots,v_{k-1})$.

The following theorems, and their equivalence, are proven in \Cref{sc:Proofs}.

\begin{theorem}\label{pr:det vol}
	Let $M=M(v_1,\ldots,v_n)$ for $v_1,\ldots,v_n \in \F^n$. 
	\begin{enumerate}[label = (\roman*)]
		\item $|\det M| = \VV_n(v_1,\ldots,v_n)$, in the real case. \label{it:real det}
		
		\item $|\det M|^2 = \VV_{2n}(v_1,\im v_1,\ldots,v_n,\im v_n)$, in the complex case. \label{it:complex det}
	\end{enumerate}
\end{theorem}

\begin{theorem}\label{pr:Gramian}
	Let $G = G(v_1,\ldots,v_p)$ for $v_1,\ldots,v_p \in \F^n$. 
	\begin{enumerate}[label = (\roman*)]
		\item $\det G = \VV_p(v_1,\ldots,v_p)^2$, in the real case. \label{it:real Gramian}
			\CITE{Shafarevich2013, p.219-221}
		
		\item $\det G = \VV_{2p}(v_1,\im v_1,\ldots,v_p,\im v_p)$, in the complex case. \label{it:complex Gramian}
	\end{enumerate}
\end{theorem}

\begin{remark}\label{rm:iP P non orthog}
	As $\im$ produces a $\frac\pi2$ rotation in $(\C^n)_\R$,
	one might erroneously think that 
	$\VV_{2p}(v_1,\im v_1,\ldots,v_p,\im v_p)$ equals $\VV_p(v_1,\ldots,v_p)^2$,
	the product of the equal volumes of $\PP = \PP(v_1,\ldots,v_p)$ and the rotated $\im \PP = \PP(\im v_1,\ldots, \im v_p)$.
	This fails as the direction of rotation might not be orthogonal to $\PP$ (see \Cref{ex:figura}). 
	\OMIT{and \Cref{pr:complex real concepts}}
\end{remark}

\begin{theorem}\label{pr:det scale}
	Let $T$ be an $\F$-linear transformation of $\F^n$.
	\begin{enumerate}[label = (\roman*)]
		\item In the real case, $T$ scales $n$-volumes in $\R^n$ by a factor of $|\det T|$. \label{it:real det scale}
		
		\item In the complex case, $T$ scales $2n$-volumes in $(\C^n)_\R$ by a factor of $|\det T|^2$.\label{it:complex det scale}
	\end{enumerate}
\end{theorem}

Relations between real determinants and linear independence or invertibility have well known interpretations in terms of degenerate parallelotopes,
which now extend to complex ones:
e.g., 
we have $\det G(v_1,\ldots,v_p) =0 \Leftrightarrow
\VV_{2p}(v_1,\im v_1,\ldots,v_p,\im v_p) = 0
\Leftrightarrow
\PP(v_1,\im v_1,\ldots,v_p,\im v_p)$ is degenerate
$\Leftrightarrow 
v_1,\im v_1,\ldots,v_p,\im v_p$ 
are $\R$-linearly dependent
$\Leftrightarrow v_1,\ldots,v_p$ are $\C$-linearly dependent.

\subsection{Proofs}\label{sc:Proofs}

We give various proofs, to provide different insights into how determinants relate to volumes,
and to show why some proofs of the real case fail in the complex one.
Equivalence of \Cref{pr:Gramian,pr:det vol,pr:det scale} is proven as usual:

\begin{proof}[Proof of Equivalence]
	(\ref{pr:Gramian} $\Rightarrow$ \ref{pr:det vol})
	If $p=n$ and $M^\dagger$ is the (conjugate if $\F=\C$) transpose of $M$ then
	$G = M^\dagger M$,
	so $\det G = |\det M|^2$.
	
	(\ref{pr:det vol} $\Rightarrow$ \ref{pr:Gramian})
	As $G$ is invariant by orthogonal (unitary if $\F=\C$) transformations, we can assume $v_1,\ldots,v_p \in \F^p$, so that $M = M(v_1,\ldots,v_p)$ is a $p \times p$ matrix and again $\det G = |\det M|^2$.
	
	(\ref{pr:det vol} $\Leftrightarrow$ \ref{pr:det scale})
	Let $T$ be given by
	$M(v_1,\ldots,v_n)$ in the canonical basis  $(e_1,\ldots,e_n)$ of $\F^n$.
	In the real case, it maps the unit box $\PP(e_1,\ldots,e_n)$ to $\PP(v_1,\ldots,v_n)$.
	In the complex one, it maps the unit box $\PP(e_1,\im e_1,\ldots,e_n,\im e_n)$ to $\PP(v_1,\im v_1,\ldots,v_n,\im v_n)$.
	As linear transformations scale equally all top dimensional volumes, the result follows.
\end{proof}

The following algebraic lemmas link real and complex determinants.
Similar to results in \cite[p.\,188]{Greub1978} and \cite[p.\,18]{Griffiths1994},%
\CITE{Griffiths1994 tem \ref{pr:lema} para Jacobianos.
	Greub1978 tem \ref{pr:det MR} com ordem diferente pois identifica $\C^n$ com $\R^{2n}$ pondo partes reais depois imaginárias}
they also follow (more abstractly) via field norms \cite[p.\,546]{Bourbaki1989}.
	\CITE{eq. (24)}
They allow easy proofs of the complex case of the theorems, as in \cite{Lahtonen2012,Malyshev2016}, but provide little geometric insight:
e.g., \Cref{pr:det vol}\ref{it:complex det} follows from the real case and \Cref{pr:det MR}.

\begin{lemma}\label{pr:lema}
	$|\det (A+\im B)|^2 = \det \begin{psmallmatrix}
		A & -B \\
		B & A
	\end{psmallmatrix}$
	for real $p \times p$ matrices $A$ and $B$.
\end{lemma}
\begin{proof}
	If
	$M = A+\im B$,
	$N = \begin{psmallmatrix}
			A & -B \\
			B & A
		\end{psmallmatrix}$
	and 
	$T = \frac{1}{\sqrt{2}} 
	\begin{psmallmatrix}
		\Id & \im \Id \\
		\im \Id & \Id
	\end{psmallmatrix}$,
	where $\Id$ is the $p \times p$ identity,
	then
	\OMIT{$T^{-1} = \frac{1}{\sqrt{2}} \cdot
		\begin{psmallmatrix}
			\Id & -\im \Id \\
			-\im \Id & \Id
		\end{psmallmatrix}$}
	$T N T^{-1} = \begin{psmallmatrix}
		M & 0 \\
		0 & \overline{M}
	\end{psmallmatrix}$
	and
	$\det N = \det(T N T^{-1}) = |\det M|^2$.
\end{proof}

\begin{lemma}\label{pr:det MR}
	$|\det M(v_1,\ldots,v_n)|^2 = \det M_\R(v_1,\im v_1,\ldots,v_n,\im v_n)$ for $v_1,\ldots,v_n \in \C^n$.
\end{lemma}
\begin{proof}
	If $v_j = (a_{1j} + \im b_{1j},\ldots,a_{nj} + \im b_{nj})$, with $a_{ij},b_{ij}\in \R$ for $i,j \in \{1,\ldots,n\}$, then
	$M = A+\im B$ for $A = (a_{ij})_{n\times n}$ and $B=(b_{ij})_{n\times n}$, and
	\begin{equation*}
		M_\R = 
		\begin{pmatrix}
			a_{11} & -b_{11} & \cdots & a_{1n} & -b_{1n} \\
			b_{11} & a_{11} & \cdots & b_{1n} & a_{1n} \\
			\vdots & \vdots & \ddots & \vdots & \vdots \\
			a_{n1} & -b_{n1} & \cdots & a_{nn} & -b_{nn} \\
			b_{n1} & a_{n1} & \cdots & b_{nn} & a_{nn} 
		\end{pmatrix}.
	\end{equation*}
	An equal number of row and column switches gives
	$\det M_\R = 
	\det \begin{psmallmatrix}
		A & -B \\
		B & A
	\end{psmallmatrix}$,
	so the result follows from \Cref{pr:lema}.
\end{proof}

For the real case of \Cref{pr:det vol}, a common proof \cite{Lang1997} uses the fact that $\VV$ has the following properties, which also characterize $|\det M|$ uniquely:
	\OMIT{the absolute value is due to \ref{it:permutation} and \ref{it:dilation}}
\begin{enumerate}[label = {(\alph*)}]
	\item $\VV(e_1,\ldots,e_n) = 1$ for the canonical basis $(e_1,\ldots,e_n)$. \label{it:unit box}
	
	\item $\VV(v_1,\ldots,v_n) =0 \Leftrightarrow$ the vectors are linearly dependent. \label{it:LD}
	
	\item $\VV(v_1,\ldots,v_n)$ does not depend on the order of the vectors. \label{it:permutation}
	
	\item $\VV(v_1+u,v_2,\ldots,v_n) = \VV(v_1,\ldots,v_n)$ for $u \in \Span\{v_2,\ldots,v_n\}$. \label{it:shear}
	
	\item $\VV(c v_1,v_2,\ldots,v_n) = |c|\cdot \VV(v_1,\ldots,v_n)$ for $c \in \R$. \label{it:dilation}
\end{enumerate}

In the complex case, \ref{it:dilation} fails as multiplication by $c\in \C$ can rotate $v_1$, changing its height \wrt the others (see \Cref{ex:figura}).
But we have:
\begin{enumerate}[label = {(\alph*')}]\setcounter{enumi}{4}
	\item $\VV(c v_1, c \im v_1, v_2,\im v_2,\ldots,v_n,\im v_n) = |c|^2\cdot \VV(v_1,\im v_1,\ldots,v_n,\im v_n)$, for $c \in \C$. \label{it:complex dilation}
\end{enumerate} 
\begin{proof}[Proof of \ref{it:complex dilation}]
	As the orthogonal projection $P_W$ on $W = \Span_\C\{v_2,\ldots,v_n\} = \Span_\R\{v_2,\im v_2, \ldots, v_n,\im v_n\}$ is $\C$-linear,
	$v_1$ and $\im v_1$ have the same height $h=\|v_1-P_W v_1\| = \|\im v_1 - P_W(\im v_1)\|$ \wrt $W$, while $c v_1$ and $c\im v_1$ have $|c|h$.
	And as $v_1 \perp_\R \im v_1$, the height of $v_1$ \wrt $W \oplus \R(\im v_1)$
	is also $h$,
	so that $\VV(v_1,\im v_1,\ldots,v_n,\im v_n) = h \cdot \VV(\im v_1,\ldots,v_n,\im v_n) = h^2 \cdot \VV(v_2,\im v_2,\ldots,v_n,\im v_n)$.
	Likewise, $\VV(c v_1,c \im v_1, v_2,\im v_2,\ldots,v_n,\im v_n) = |c|^2 \cdot h^2 \cdot \VV(v_2,\im v_2,\ldots,v_n,\im v_n)$.
\end{proof}

One can prove \Cref{pr:det vol}\ref{it:complex det} by checking that
$\VV(v_1,\im v_1,\ldots,v_n,\im v_n)$, like $|\det M|^2$, satisfies
\ref{it:unit box}--\ref{it:shear}, with adequate adjustments%
	\footnote{Such as: (a') $\VV(e_1,\im e_1,\ldots,e_n,\im e_n) = 1$ for the canonical basis $(e_1,\ldots,e_n)$ of $\C^n$.},
and \ref{it:complex dilation}.
A more direct proof uses shears to turn the parallelotope into an orthogonal box:

\begin{proof}[Proof of \Cref{pr:det vol}\ref{it:complex det}]
	$\VV(v_1,\im v_1,\ldots,v_n,\im v_n) = h^2 \cdot \VV(v_2,\im v_2,\ldots,v_n,\im v_n)$, as above, and $h=\|u_1\|$ for $u_1=v_1-\sum_{j=2}^n c_{1j} v_j$ with $c_{1j}\in \C$ such that $\inner{u_1,v_j} = 0$ for $j > 1$.
	Using induction, we find 
	$\VV(v_1,\im v_1,\ldots,v_n,\im v_n) = \prod_{k=1}^n \|u_k\|^2$, where $u_k = v_k-\sum_{j=k+1}^n c_{kj} v_j$ with $c_{kj}\in \C$ such that $\inner{u_k,v_j} = 0$ for $j > k$.
	And we also have 
	$|\det M(v_1,\ldots,v_n)|^2 = |\det M(u_1,\ldots,u_n)|^2 = 
	\det G(u_1,\ldots,u_n) = \prod_{k=1}^n \|u_k\|^2$
	as $(u_1,\ldots,u_n)$ is $\C$-orthogonal.
\end{proof}

The next proof is a little more geometric, using shears given by elementary column operations instead of the algebraic relation $|\det M|^2 = \det G$.

\begin{proof}[Proof of \Cref{pr:det vol}\ref{it:complex det}]	
	We can assume the vectors are $\C$-linearly independent, as
	$\det M = 0 \Leftrightarrow v_1,\ldots,v_n$ are $\C$-linearly dependent $\Leftrightarrow v_1,\im v_1,\ldots,v_n,\im v_n$ are $\R$-linearly dependent $\Leftrightarrow 
	\VV_{2n}(v_1,\im v_1,\ldots,v_n,\im v_n) = 0$.
	
	Reordering the vectors and repeatedly adding to one a $\C$-multiple of another
		\OMIT{as in a Gaussian elimination}
	we can turn $M$ into  
	$\tilde{M} = \diag(\lambda_1,\ldots,\lambda_n)$ with 
	$\lambda_i \in \C$, and 
	$|\det M| = |\det \tilde{M}|$. 
	This preserves $\VV_{2n}(v_1,\im v_1,\ldots,v_n,\im v_n)$, as, for example, $v_1 \mapsto v_1 + (a+\im b) v_2$, with $a,b\in \R$, corresponds to shears $v_1 \mapsto v_1 + a v_2 + b (\im v_2)$ and $\im v_1 \mapsto \im v_1 - b v_2 + a (\im v_2)$.
	If $\tilde{v}_1,\ldots,\tilde{v}_n$ are the columns of $\tilde{M}$ then $\tilde{v}_1,\im \tilde{v}_1, \ldots,\tilde{v}_n, \im \tilde{v}_n$ are $\R$-orthogonal, forming a box of volume $\VV_{2n} = \prod_{j=1}^n \|\tilde{v}_j\| \|\im \tilde{v}_j\| = \prod_{j=1}^n |\lambda_j|^2 = |\det \tilde{M}|^2$.
\end{proof}

Now we prove \Cref{pr:Gramian}, including a proof of the real case which has been erroneously used in the complex one, as we discuss below.

\begin{proof}[Proof of \Cref{pr:Gramian}]
	\ref{it:real Gramian} 
	Follows via induction:
	$\det G(v_1) = \|v_1\|^2 = \VV(v_1)^2$,
	and, with the notation of \Cref{df:main}\ref{it:volume}, subtracting $c_j$ times each column $j < p$ from the last column of $G(v_1,\ldots,v_p)$, 
	and using $\inner{v_j,u_p}=0$, we find
	\begin{align*}
		\det G(v_1,\ldots,v_p) &= 
		\begin{vmatrix}
			\inner{v_1,v_1} & \cdots & \inner{v_1,v_{p-1}} & 0 \\
			\vdots & \ddots & \vdots & \vdots \\
			\inner{v_{p-1},v_1} & \cdots & \inner{v_{p-1},v_{p-1}} & 0 \\
			\inner{v_p,v_1} & \cdots & \inner{v_p,v_{p-1}} & \inner{v_p,u_p}
		\end{vmatrix} \\
		&= \det G(v_1,\ldots,v_{p-1}) \cdot \inner{v_{p},u_p} \\
		&= \VV(v_1,\ldots,v_{p-1})^2 \cdot \|u_p\|^2 = \VV(v_1,\ldots,v_p)^2.
	\end{align*}
	
	\ref{it:complex Gramian}
	Let $A = (a_{ij})_{p\times p}$ and $B=(b_{ij})_{p\times p}$
	with
	\SELF{$\inner{v_i,v_j} = a_{ij} + \im b_{ij}$}
	$a_{ij} = \inner{v_i,v_j}_\R = \inner{\im v_i,\im v_j}_\R$ 
	and 
	$b_{ij} = \inner{\im v_i, v_j}_\R = -\inner{v_i, \im v_j}_\R$.
	Then
	$G = G(v_1,\ldots,v_p) = A + \im B$
	and
	$G_\R = G_\R(v_1,\ldots,v_p, \im v_1,\ldots,\im v_p) = \begin{psmallmatrix}
		A & -B \\
		B & A
	\end{psmallmatrix}$.
	By \Cref{pr:lema} and \ref{it:real Gramian},
	$|\det G|^2 = \det G_\R = \VV(v_1,\ldots,v_p, \im v_1,\ldots,\im v_p)^2$, and the result follows as $\det G \geq 0$.	
\end{proof}

Part (i) fails in the complex case because in \Cref{df:main}\ref{it:volume} we have $\inner{v_j,u_k}_\R = 0$.
Some authors \cite{Barth1999,Gantmacher1959,Halperin1962}
	\CITE{Gantmacher1959 ou 2000 pp.246-251}
define $\VV$ using $\inner{v_j,u_k} = 0$ or in other ways that give $\det G = \VV^2$ even in the complex case.
This is geometrically incorrect: e.g., $\VV_2(u,v)$ would be $\|u\|$ times the height of $v$ \wrt $\C u$, not $\R u$, so the area in \Cref{fig:exemplo1} would be $6$, 
and the square $S$ in \Cref{ex:square} would have area $0$.

Our last proof, of \Cref{pr:det scale}, gives a clear geometric picture, even if it uses facts about matrices which are not so geometrically intuitive.

\begin{proof}[Proof of \Cref{pr:det scale}]
	We start with the complex case, which is simpler.
	
	\ref{it:complex det scale}
	The matrix of $T$ in the canonical basis $(e_1,\ldots,e_n)$ of $\C^n$ decomposes as $T = R^{-1} J R$, with $R$ invertible and $J$ in Jordan normal form.
	As $R^{-1}$ and $R$ scale $2n$-volumes in $(\C^n)_\R$ by inverse factors, $T$ and $J$ scale them equally.
	Each $\lambda_j \in \C$ in the diagonal of $J$ scales $e_j$ and $\im e_j$ by $|\lambda_j|$,
	and 
		\OMIT{if $\lambda_j \neq 0$}
	rotates
		\OMIT{As the basis is orthonormal, which would not be the case if we started with the basis in which $T$ is given by $J$.}	
	$\Span_\R\{e_j,\im e_j\}$ without further affecting $2n$-volumes.
	These are also not affected by shears given by
	superdiagonal elements, and so scale by $\prod_{j=1}^n |\lambda_j|^2 = |\det T|^2$.

	\ref{it:real det scale}
	Similar, but the diagonal of a real Jordan form can have eigenvalues $\lambda_j \in \R$, which scale $e_j$ by $|\lambda_j|$ and maybe reflect it,
	or, for each pair of complex conjugate $\lambda_j$ and $\lambda_{j+1}$, a $2\times 2$ block which rotates $\Span\{e_j,e_{j+1}\}$ and scales its vectors by $|\lambda_j| = |\lambda_{j+1}|$.
	Thus $n$-volumes scale by $\prod_{j=1}^n |\lambda_j| = |\det T|$.
\end{proof}

\subsection{Examples}

The following examples will be further developed later on.
Note how calculations in $\C^n$ are simpler than in the underlying $\R^{2n}$.

\begin{figure}
	\centering
	\begin{subfigure}[b]{0.37\textwidth}
		\includegraphics[width=\textwidth]{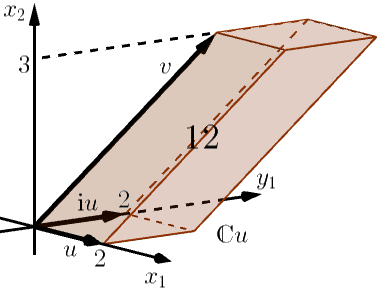}
		\caption{$\VV_3(u,\im u,v)=12$}
		\label{fig:exemplo3}
	\end{subfigure}
	\begin{subfigure}[b]{0.29\textwidth}
		\includegraphics[width=\textwidth]{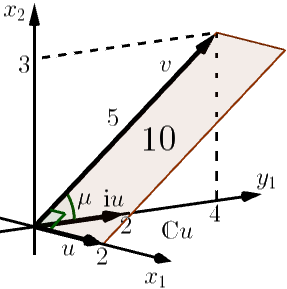}
		\caption{$\VV_2(u,v)=10$}
		\label{fig:exemplo1}
	\end{subfigure}
	\begin{subfigure}[b]{0.31\textwidth}
		\includegraphics[width=\textwidth]{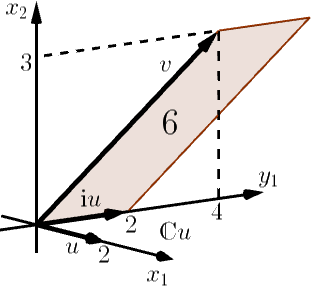}
		\caption{$\VV_2(\im u,v)=6$}
		\label{fig:exemplo2}
	\end{subfigure}
	\caption{Volumes of parallelotopes spanned by vectors of \Cref{ex:figura}.}
	\label{fig:exemplos}
\end{figure}

\begin{example}\label{ex:figura}%
		\OMIT{$u,v \in\C^2$, Figura, pure real, \Cref{ex:figura,ex:orientations,ex:figura3,ex:index,ex:figura6}}%
	In $\C^2 = \{(x_1+\im y_1,x_2+\im y_2)\}$, 
	let $u=(2,0)$ and $v=(4\im,3)$. 
	In the underlying $\R^4=\{(x_1,y_1,x_2,y_2)\}$, 
	$u=(2,0,0,0)$, $\im u = (0,2,0,0)$, $v=(0,4,3,0)$ and $\im v = (-4,0,0,3)$.
	\Cref{fig:exemplos} shows $u$, $\im u$, $v$ in $\R^3$, ignoring $y_2=0$.
	As $\im v$ has height $3$ \wrt $\R^3$, 
	$\VV_4(u,\im u,v,\im v)= 3\cdot\VV_3(u,\im u,v) = 36$, by \Cref{fig:exemplo3}. 
	In agreement with \Cref{pr:det vol}, $\det M(u,v)  = \begin{vsmallmatrix}
		2 & 4\im \\
		0 & 3
	\end{vsmallmatrix} = 6$, and 
	\begin{equation*}
		\det M_\R(u,\im u,v,\im v) = 
		\begin{vmatrix}
			2 & 0 & 0 & -4 \\
			0 & 2 & 4 & 0 \\
			0 & 0 & 3 & 0 \\
			0 & 0 & 0 & 3
		\end{vmatrix}
		= 36.
	\end{equation*}
	By \Cref{pr:Gramian}, $\VV_2(u,v) = \sqrt{G_\R(u,v)} =
	\begin{vsmallmatrix}
		4 & 0 \\
		0 & 25
	\end{vsmallmatrix}^\frac12 = 10$,
	as in \Cref{fig:exemplo1}.
	\Cref{fig:exemplo2} shows $\VV_2(\im u,v) \neq |\im|\cdot \VV_2(u, v)$, so property \ref{it:dilation} fails.
	As in \Cref{rm:iP P non orthog}, $\VV_4(u,\im u,v,\im v) \neq \VV_2(u,v)^2$ because $\PP(u,v)$ and $\PP(\im u,\im v)$ form an angle $\mu \neq \frac\pi2$.
\end{example}

\begin{example}\label{ex:totally real0}%
		\OMIT{$u,v\in\C^2$, tot real, \Cref{ex:totally real0,ex:orientations,ex:totally real}}%
	In $\C^2$, $u=(5\im,0)$ and $v=(0,4\im)$ form a rectangle of area $\VV_2(u,v) = 20$, which in this case coincides with $|\det M(u,v)|$.
	In the un\-der\-ly\-ing $\R^4$, $u=(0,5,0,0)$, $\im u = (-5,0,0,0)$, $v= (0,0,0,4)$ and $\im v = (0,0,-4,0)$ are orthogonal,	so
	$\PP(u,\im u,v,\im v)$ is a box with $\VV_4(u,\im u,v,\im v) = 400 = \det M_\R(u,\im u,v,\im v)$.
\end{example}

\begin{example}\label{ex:square}%
		\OMIT{$u,v \in \C$, square $S$, holom, só $u$ tot real, \Cref{ex:square,ex:orientations,ex:square2} (\ref{ex:square4} removido)}%
	In $\C$, $u=2$ and $v=2\im$ are linearly dependent,
	$\det G(u,v) = \begin{vsmallmatrix}
		4 & 4\im \\
		-4\im & 4
	\end{vsmallmatrix} = 0$,
	and 
	$\PP(u, \im u, v, \im v) = \PP(2,2\im,2\im,-2)$ is degenerate, with $\VV_4=0$.
	In the underlying $\R^2$, $u=(2,0)$ and $v=(0,2)$ form a square $S$ of area $\VV_2(u,v) = \sqrt{\det G_\R (u,v)} = \begin{vsmallmatrix}
		4 & 0 \\
		0 & 4
	\end{vsmallmatrix}^\frac12 = 4$.
	Equivalently,
	$\VV_2(u,v) = \VV_2(u,\im u) = \det G(u)= 4$.
\end{example}

\begin{example}\label{ex:3 em C3 parte0}%
		\OMIT{$u,v,w\in \C^3$, \Cref{ex:3 em C3 parte0,ex:orientations,ex:3 em C3 parte1,ex:index}}
	Let $u,v,w \in \C^3$ be the columns of
	$M =
	\begin{psmallmatrix}
		1 & 1+\im & 1+2\im \\
		0 & 1 & -1 \\
		-\im & 0 & \im
	\end{psmallmatrix}$.
	As $\det M = 3\im -3$,
	$\VV_6(u,\im u,v,\im v,w,\im w) = |\det M|^2 =  18$, the same as the $6 \times 6$ determinant of $M_\R(u,\im u,v,\im v,w,\im w)$.
	And $\VV_3(u,v,w) = \sqrt{\det G_\R(u,v,w)} =
	\begin{vsmallmatrix}
		2 & 1 & 0 \\
		1 & 3 & 2 \\
		0 & 2 & 7
	\end{vsmallmatrix}^\frac12
	= \sqrt{27}$. 
\end{example}

\begin{example}\label{ex:C3}%
		\OMIT{$u,v\in\C^3$, G, \Cref{ex:orientations,ex:C3 nova}}%
	In $\C^3$, 
	if $u = (1+2\im, 0, 3\im)$ and $v=(2,\im,3+\im)$ then $\VV_2(u,v) = \sqrt{\det G_\R (u,v)} =
	\begin{vsmallmatrix}
		14 & 5 \\
		5 & 15 
	\end{vsmallmatrix}^\frac12
	= \sqrt{185}$.
	And
	$\VV_4(u,\im u,v,\im v) = \det G(u,v) = \begin{vsmallmatrix}
		14 & 5-13\im \\
		5+13\im & 15
	\end{vsmallmatrix} = 16 = \sqrt{\det G_\R(u,\im u,v,\im v)}$ for
	\begin{equation*}
		G_\R (u,\im u,v,\im v) = 
		\begin{pmatrix}
			14 & 0 & 5 & \ 13 \\
			0 & 14 & -13 & \ 5 \\
			5 & -13 & 15 & \ 0 \\
			13 & 5 & 0 & \ 15
		\end{pmatrix}.
	\end{equation*}
\end{example}

\begin{example}\label{ex:complex T}%
		\OMIT{$T$ in $\C^2$, sem 0s, \Cref{ex:complex T,ex:orientations,ex:complex T 2}, (\ref{ex:blade norms} removido)}
	The $\C$-linear transformation of $\C^2$ given by
	$T = \begin{psmallmatrix}
		\im - \sqrt{3} & 1+2\im \\
		1+\im \sqrt{3} & -\im
	\end{psmallmatrix}$
	has $\det T = 4 e^{-\im\frac\pi6}$.
	By \Cref{pr:det scale}, it expands 4-volumes by $|\det T|^2 = 16$.
	Indeed, the corresponding $\R$-linear transformation $T_\R$ of the underlying $\R^4$ has%
	\begin{equation*}
		\det T_\R = \begin{vmatrix}
			-\sqrt{3} & -1 & 1 & \ -2 \\
			1 & -\sqrt{3} & 2 & \ 1 \\
			1 & -\sqrt{3} & 0 & \ 1 \\
			\sqrt{3} & 1 & -1 & \ 0
		\end{vmatrix} = 16.
	\end{equation*}
\end{example}

\section{Orientations, arguments and  transformations}\label{sc:Complex orientation}

To interpret the argument of complex determinants we need complex orientations, which are unorthodox but useful \cite{Mandolesi_Contractions}.
For a subspace $V \subset \F^n$, let
$GL(V)$, $U(V)$ and $SU(V)$ be its general linear, unitary and special unitary groups%
\footnote{In the real case, $U(V)$ and $SU(V)$ are the orthogonal groups $O(V)$ and $SO(V)$.}, and
$GL^+(V) = \{T\in GL(V):\det T>0\}$.
Bases are ordered.

\begin{definition}\label{df:orientation}
	The set $\OO_V$ of \emph{orientations} of $V \subset \F^n$ can be defined in the following equivalent%
		\footnote{Equivalence in the complex case follows as in the usual real one.}
	ways:
	\begin{enumerate}[label = (\roman*)]
		\item $\OO_V = \{$bases $\beta$ of $V\}/\!\sim$, with $\beta_1 \sim \beta_2$ if $\beta_2 = T(\beta_1)$ for $T \in GL^+(V)$;\label{it:orientation bases}
		
		\item $\OO_V = \{$orthonormal bases $\beta$ of $V\}/\!\sim$, with $\beta_1 \sim \beta_2$ if $\beta_2 = T(\beta_1)$ for $T \in SU(V)$. \label{it:orientation orthon bases}
%
	\end{enumerate}
	The \emph{canonical orientation} of $\F^n$ is that of its canonical basis.
\end{definition}

This definition is common in the real case \cite{Shafarevich2013},
\CITE{Shaw p.348}
and extends naturally to the complex one, but the result has an important difference.
Fixed a basis $\beta_0$, any other is $\beta = T(\beta_0)$ for $T\in GL(V)$.
If $\det T = r e^{\im \varphi}$ ($r>0$) 
then $T = T' U$ with $U = \diag(1,\ldots,1,e^{\im \varphi})$ and $T' \in GL^+(V)$.
So \ref{it:orientation bases} and \ref{it:orientation orthon bases} give bijective correspondences
$\OO_V \simeq GL(V)/GL^+(V) \simeq U(V)/SU(V) \simeq \{\diag(1,\ldots,1,e^{\im \varphi})\} \simeq \{e^{\im \varphi} : \varphi = \arg(\det T) \text{ for } T \in GL(V)\}$.
In the real case,
$\varphi = 0$ or $\varphi = \pi$, so $\OO_V \simeq \{\pm 1\}$ (unit circle in $\R$) has only 2 real orientations, as usual.
The complex case admits any $\varphi\in [0,2\pi)$, so $\OO_V$ identifies with the unit circle in $\C$, and there is a continuum of complex orientations.
We call $\varphi$ the \emph{phase difference} between the orientations of $\beta$ and $\beta_0$.

Another way to look at this is to consider, for orthonormal  $v_1,\ldots, v_{p-1} \in V$, with $p= \dim V$, the choices for $v_p$ to complete an orthonormal basis. If $\F=\R$, given one such $v_p$ the only other choice is $-v_p$, and each gives a real orientation.
If $\F=\C$, any $e^{\im \varphi} v_p$ works, and each gives a complex orientation.%
\SELF{complex orientation of $(v_1,\ldots,v_n)$ is NOT linked to real ones of $(v_1,\im v_1,\ldots,v_n,\im v_n)$, which is $+$, or $(v_1,\ldots,v_n,\im v_1,\ldots,\im v_n)$, which has constant sign.}

As bases give isomorphisms $f:\F^p \rightarrow V$,
orientations can also be seen as equivalence classes of isomorphisms (\resp isometric isomorphisms) modulo $GL^+(V)$ (\resp $SU(V)$).
For a real line, each orientation is a class of identifications with $\R$, with the positive semi-axis in one of 2 directions (\Cref{fig:real_orientations}).
For a complex line, each orientation is a class of identifications with $\C$, with the positive real semi-axis towards a point in the unit circle, and the imaginary one rotated $90^\circ$ by $\im$ (\Cref{fig:complex_orientations}).

\begin{figure}
	\centering
	\begin{subfigure}[b]{0.46\textwidth}
		\centering
		\includegraphics[width=0.95\linewidth]{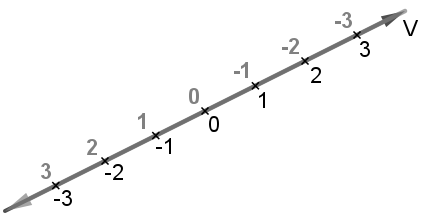}
		\caption{Opposite real orientations of a real line $V \subset \R^n$, differing by a reflection.}
		\label{fig:real_orientations}
	\end{subfigure} 
	\hspace{5mm}
	\begin{subfigure}[b]{0.46\textwidth}
		\centering
		\includegraphics[width=0.95\linewidth]{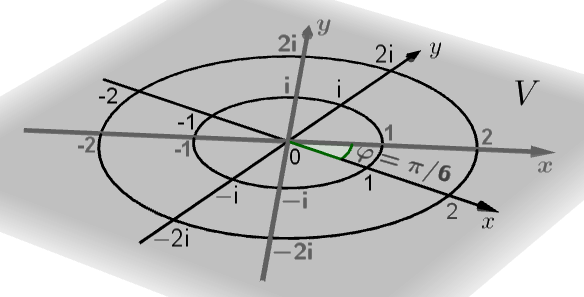}
		\caption{Complex orientations of a complex line $V \subset \C^n$, differing by a phase $\varphi=\frac\pi6$.}
		\label{fig:complex_orientations}
	\end{subfigure}
	\caption{Orientations of real and complex lines, seen as identifications with $\R$ or $\C$.}
	\label{fig:line}
\end{figure}

By \Cref{pr:det MR},
\OMIT{$\det M_\R \geq 0$. \\ Let $p=n$, w.l.o.g.}
all bases of $V_\R$ of the form $(v_1,\im v_1,\ldots,v_p,\im v_p)$, for a basis $(v_1,\ldots,v_p)$ of $V$, have the same real orientation,
considered a \emph{natural orientation} of complex spaces \cite{Greub1978,Griffiths1994}.
\CITE{Greub1978 p.188, Griffiths1994 p.18}
But, being a real orientation in $V_\R$, it does not reflect the complex geometry of $V$ as well as complex orientations (which is why these give a simpler star operator than the Hodge one \cite{Mandolesi_Contractions}).
A possible reason why complex orientations have been neglected is that they form a continuum, being inadequate for applications that need discreteness (e.g., in topology).
And they lose a nice feature: real orthonormal bases have same real orientation
$\Leftrightarrow$ they differ by a rotation
$\Leftrightarrow$ one can be continuously deformed into another.
This fails in the complex case, as even a reflection $(v_1,v_2,\ldots,v_p) \mapsto (-v_1,v_2,\ldots,v_p)$ results from continuous rotations $(e^{\im \varphi} v_1,v_2,\ldots,v_p)$, with $0\leq\varphi\leq\pi$.
In fact, any $U\in U(V)$ is in $V_\R$ a rotation (commuting with the complex structure), which is why it preserves the natural real orientation.

The action of $T \in GL(V)$ on the bases of $V$ descends to an action on $\OO_V$,
	\OMIT{$\beta = P \beta_0$ with $\det P>0 \Rightarrow T \beta = (TPT^{-1})(T \beta_0)$, with $\det(TPT^{-1})>0$}
whose identification with a circle allows a first interpretation for $\arg(\det T)$:

\begin{proposition}\label{pr:arg rotates}
	$T \in GL(V)$ rotates%
		\footnote{If $\F=\R$, the ``rotation'' is by $\varphi = 0$ ($T$ preserves orientations) or $\varphi = \pi$ ($T$ inverts them).}
	$\OO_V$ by $\varphi = \arg(\det T)$.
		\SELF{In the complex case, $\varphi = \sum_{j=1}^{\dim V} \arg(\lambda_j) \pmod{2\pi}$, where the $\lambda_j$'s are the eigenvalues of $T$ (counting multiplicities).}
\end{proposition}
\begin{proof}
	Fixed a basis $\beta_0$, the orientation of another $\beta = T_0(\beta_0)$, for $T_0\in GL(V)$, is identified, as above, with $e^{\im \varphi_0}$ for $\varphi_0 = \arg(\det T_0)$.
	And the orientation of $T(\beta) =  T(T_0(\beta_0))$ corresponds to $e^{\im \arg(\det(TT_0))}= e^{\im (\varphi + \varphi_0)}$.
\end{proof}

\begin{corollary}
	If $\det M(v_1,\ldots,v_n) \neq 0$, for $v_1,\ldots,v_n \in \C^n$, its argument is the phase difference from the canonical orientation to that of $(v_1,\ldots,v_n)$.
	\OMIT{The canonical basis and $(v_1,\ldots,v_n)$ are related by the linear transformation given by $M$.}
\end{corollary}

\begin{example}
	Given bases $\beta_0 = (v_1, v_2, \ldots,v_p)$,
	$\beta_1 = (e^{\im \varphi} v_1, v_2, \ldots, v_p)$ 
	and
	$\beta_2 = (e^{\im \frac{\varphi}{p}} v_1, e^{\im \frac{\varphi}{p}} v_2, \ldots,e^{\im \frac{\varphi}{p}} v_p)$ 
	of $V \subset \C^n$,
	$\beta_1$ and $\beta_2$
	have the same complex orientation, which differs from that of $\beta_0$ by the phase $\varphi$.
\end{example}

\subsection{Elementary transformations}
\label{Geometry of orientations and arguments}

A more geometric interpretation for $\arg(\det T)$ involves decomposing $T$.

For $V \subset \R^n$, recall that any $T\in GL(V)$ is a composition of line reflections%
\footnote{More precisely, reflections of a line (and its parallels) across a hyperplane, usually called \emph{hyperplane reflections}. As we generalize to complex and oblique reflections, referring to the line whose vectors are reflected (sent to their opposites) feels more adequate than to the hyperplane of fixed points (which, in the oblique case, does not act as a mirror).}, 
line scalings, shears, plane rotations, and transpositions of basis vectors.
The last two are included for convenience:
rotations are composed of shears and scalings, or pairs of line reflections,
and a transposition $v_j \leftrightarrow v_k$ is a reflection of $v_j - v_k$.
	\OMIT{$(v_1-v_2,v_1+v_2,v_3,\ldots,v_p) \mapsto (v_2-v_1,v_1+v_2,v_3,\ldots,v_p)$}
Also, $\sgn(\det T) = (-1)^N$ for $N =$ number of line reflections and transpositions, and there is always a decomposition with $N=0$ or $N=1$.

Being interested in volumes, not isometries, 
we define \emph{generalized line reflections and plane rotations} (\Cref{fig:rot_refl}) as linear transformations mapping a not necessarily orthonormal basis 
$(v_1,\ldots,v_p)$ of $V$ respectively to 
$(-v_1,v_2,\ldots,v_p)$ and
$(v_1 \cos \theta + v_2 \sin \theta, -v_1 \sin \theta + v_2 \cos \theta, v_3,\ldots,v_p)$, for $\theta \in [0,2\pi)$.
They preserve $p$-volumes,
	\OMIT{\ref{pr:det scale}}
being compositions of (orthogonal) reflections and rotations with shears (used to make the basis vectors orthogonal) and scalings by inverse factors (to make their norms equal).

\begin{figure}
	\centering
	\begin{subfigure}[b]{0.35\textwidth}
		\centering
		\includegraphics[width=\textwidth]{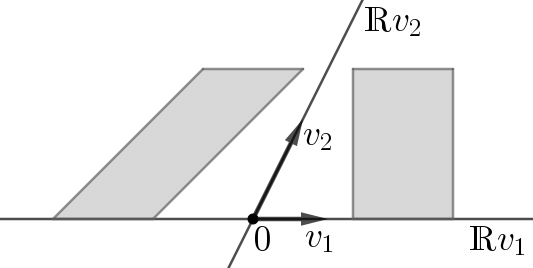}
		\caption{Generalized line reflection $(v_1,v_2) \mapsto (-v_1,v_2)$ of $\R v_1$, fixing the non-orthogonal $\R v_2$. The image of the rectangle is slanted, but has the same area. \\ }
		\label{fig:reflection}
	\end{subfigure} 
	\hspace{5mm}
	\begin{subfigure}[b]{0.58\textwidth}
		\centering
		\includegraphics[width=\textwidth]{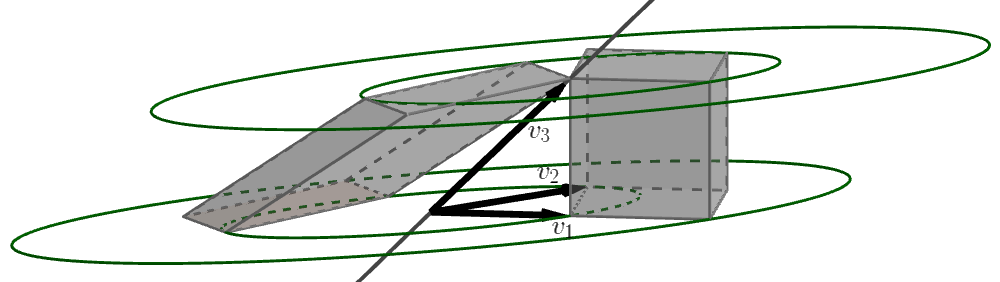}
		\caption{Generalized plane rotation $(v_1,v_2,v_3) \mapsto (v_1 \cos \theta + v_2 \sin \theta, -v_1 \sin \theta + v_2 \cos \theta, v_3)$. 
		The fixed ``rotation axis'' $\R v_3$ is not orthogonal to the ``rotation plane'' $\Span_\R\{v_1,v_2\}$,
		and orbits for $0\leq\theta\leq 2\pi$ are ellipses. The image of the cube is stretched by inverse factors and slanted, keeping the same volume.}
		\label{fig:rotation}
	\end{subfigure}
	\caption{Generalized line reflection and plane rotations, for non-orthonormal bases.}
	\label{fig:rot_refl}
\end{figure}

For $V \subset \C^n$, we use the following \emph{elementary complex transformations}:

\begin{definition} 
	If $T \in GL(V)$ maps a basis $(v_1,v_2,\ldots,v_p)$ of $V \subset \C^n$ to:
	\begin{enumerate}[label = (\alph*)]
		\item $(v_1+u,v_2,\ldots,v_p)$, for $u \in \Span_\C\{v_2,\ldots,v_p\}$, it is a \emph{complex shear}. \label{it:complex shear}				
		
		\item $(c v_1,v_2,\ldots,v_p)$, for $c > 0$, it is a \emph{complex line scaling}. \label{it:complex scaling}
		
		\item $(e^{\im \varphi} v_1,v_2,\ldots,v_p)$, for $\varphi \in [0,2\pi)$, it is a \emph{phase rotation}. \label{it:phase rotation}
		
		\item $(-v_1,v_2,\ldots,v_p)$, it is a \emph{generalized complex line reflection}. \label{it:complex reflection}
		
		\item $(v_2,v_1,v_3,\ldots,v_p)$, it is a \emph{complex transposition}. \label{it:complex transposition}
		
		\item $(v_1 \cos \theta + v_2 \sin \theta, -v_1 \sin \theta + v_2 \cos \theta, v_3,\ldots,v_p)$, for $\theta \in [0,2\pi)$,
		it is a \emph{generalized complex plane rotation}. \label{it:totally real rotation}
	\end{enumerate}
\end{definition}

The last three are a convenience:
\ref{it:complex reflection} and \ref{it:complex transposition}
are particular cases%
	\footnote{With $\varphi=\pi$, and $(v_1 - v_2,v_1+v_2,v_3,\ldots,v_p) \mapsto (e^{\im \pi}(v_1 - v_2),v_1+v_2,v_3,\ldots,v_p)$.}
of \ref{it:phase rotation},
and 
\ref{it:totally real rotation} is a composition of complex shears and scalings.
\Cref{pr:arg rotates} shows 
\ref{it:complex shear}, \ref{it:complex scaling} and \ref{it:totally real rotation} preserve complex orientations,
\ref{it:phase rotation} rotates $\OO_V$ by the phase $\varphi$, 
while \ref{it:complex reflection} and \ref{it:complex transposition} rotate it by $\pi$.
These $\C$-linear transformations move $v$ and $\im v$ in tandem, corresponding in $V_\R$ to the following real transformations (note that they all preserve real orientations):
\begin{enumerate}[label = (\roman*)]
	\item[\ref{it:complex shear}] 2 shears,  $(v_1,\im v_1,\ldots,v_p,\im v_p) 
	\mapsto (v_1+ u,\im v_1+ \im u,\ldots,v_p,\im v_p)$.
	
	\item[\ref{it:complex scaling}] 2 line scalings, of $\R v_1$ and $\R(\im v_1)$, by $c$ (so areas in $\C v_1$ scale by $c^2$).
	
	\item[\ref{it:phase rotation}] A generalized%
		\footnote{Restricted to $\Span_\R\{v_1,\im v_1\}$, it is a usual (orthogonal) rotation, but not on the whole $V_\R$.}
	rotation by $\varphi$ of the real plane $\Span_\R\{v_1,\im v_1\}$.
	
	\item[\ref{it:complex reflection}] 2 generalized line reflections, of $\R v_1$ and $\R(\im v_1)$, corresponding to a generalized rotation by $\pi$ of $\Span_\R\{v_1,\im v_1\}$.
		 	
	\item[\ref{it:complex transposition}] 2 transpositions, $v_1 \leftrightarrow v_2$ and $\im v_1 \leftrightarrow \im v_2$.
	
	\item[\ref{it:totally real rotation}] 2 generalized rotations by $\theta$ of the real planes $U = \Span_\R\{v_1,v_2\}$ and $\im U = \Span_\R\{\im v_1,\im v_2\}$ inside the complex plane $\Span_\C\{v_1,v_2\}$.
		\SELF{composition of complex shears and complex line scalings?}
\end{enumerate}

We can now interpret arguments of determinants as total phase rotations:

\begin{theorem}\label{pr:decomposition}
	Let $\beta = (v_1,\ldots,v_p)$ be a basis of $V \subset \C^n$ in which $T \in GL(V)$ has Jordan normal form, with eigenvalues $\lambda_1,\ldots,\lambda_p$ in its diagonal. 
	Then:
	\begin{enumerate}[label = (\roman*)]
		\item $T$ is a composition of complex shears, line scalings and phase rotations by $\varphi_j = \arg(\lambda_j)$ of the $\C v_j$'s,
		and $\arg(\det T) = \sum_{j=1}^p \varphi_j \pmod{2\pi}$. \label{it:T phases}
		
		\item $T$ is a composition of complex shears, line scalings and a single phase rotation by $\varphi = \arg(\det T)$. \label{it:T single phase}
		
		\item $T \in GL^+(V) \Leftrightarrow T$ is a composition of complex shears and line scalings. \label{it:GL+}
	\end{enumerate}
\end{theorem}
\begin{proof}
	\ref{it:T phases} 
	Each $\lambda_j$ rotates $\C v_j$ by $\varphi_j$ and scales it by $|\lambda_j|$,
	any superdiagonal $1$ causes a complex shear, and $\arg(\det T) = \arg\big(\prod_{j=1}^p \lambda_j\big) = \sum_{j=1}^p \varphi_j \pmod{2\pi}$.
	
	\ref{it:T single phase} The phase rotations in \ref{it:T phases} are produced by $R = \diag(e^{\im \varphi_1},\ldots,e^{\im \varphi_p})$ in the basis $\beta$. If $p=2$ we have, with $a = e^{\im \varphi_1}$,
	\begin{equation*}
		\begin{pmatrix}
			e^{\im \varphi_1} & 0 \\
			0 & e^{\im \varphi_2}
		\end{pmatrix} =
		\begin{pmatrix}
			1 & 0 \\
			0 & e^{\im (\varphi_1+\varphi_2)}
		\end{pmatrix}
		\begin{pmatrix}
			1 & 1-a \\
			0 & 1
		\end{pmatrix}
		\begin{pmatrix}
			1 & 0 \\
			-1 & 1
		\end{pmatrix}
		\begin{pmatrix}
			1 & 1-\bar{a} \\
			0 & 1
		\end{pmatrix}
		\begin{pmatrix}
			1 & 0 \\
			a & 1
		\end{pmatrix},	
	\end{equation*}
	so $R$ consists of complex shears and a single phase rotation of $\C v_2$ by $\varphi_1+\varphi_2$.
	If $p>2$, doing the same 2 phases at a time we can move them down the diagonal to obtain $R = \diag(1,\ldots,1,e^{\im\varphi}) \circ S$, where $\varphi = \sum_{j=1}^p \varphi_j = \arg(\det T)$
	and $S$ is a composition of complex shears.
	
	\ref{it:GL+} Follows from \ref{it:T single phase}.
\end{proof}


\begin{theorem}
	The canonical basis $\beta_0 = (e_1,\ldots,e_n)$ of $\C^n$ can be turned into any other $\beta = (v_1,\ldots,v_n)$ via \ref{it:complex shear}--\ref{it:totally real rotation}, 
		\SELF{\ref{it:totally real rotation} is not used, being included for convenience.}
	and $\arg(\det M(v_1,\ldots,v_n))$ is the total phase rotation used (mod $2\pi$, and including $\pi$ for each use of \ref{it:complex reflection} or \ref{it:complex transposition}).
\end{theorem}
\begin{proof}
	$v_1 = \sum_{j=1}^n \lambda_{1j} e_j$ with 
		\OMIT{$\lambda_{1j} \in \C$}
	$\lambda_{11} = c_1 e^{\im \varphi_1}$ for $c_1 > 0$, possibly after a transposition of $\beta_0$.
	Scaling $\C e_1$ by $c_1$, rotating it by $\varphi_1$, and applying a shear by $u_1 = \sum_{j=2}^n \lambda_{1j} e_j$, we obtain a basis $\beta_1 = (v_1,e_2,\ldots,e_n)$.
	As $\{v_1,v_2\}$ is linearly independent, $v_2 = \lambda_{21} v_1 + \sum_{j=2}^n \lambda_{2j} e_j$ with $\lambda_{22} = c_2 e^{\im \varphi_2}$ for $c_2 > 0$, possibly after a transposition of $\beta_1$.
	Scaling $\C e_2$ by $c_2$, rotating it by $\varphi_2$, and applying a shear by $u_2 = \lambda_{21} v_1 + \sum_{j=3}^n \lambda_{2j} e_j$, we turn $\beta_1$ into $\beta_2 = (v_1,v_2,e_3,\ldots,e_n)$.
	Proceeding like this, we can obtain $\beta$.
	As $M(v_1,\ldots,v_n)$ is a composition of all transformations used,
	$\arg(\det M)$ is (mod $2\pi$) the sum of $\varphi_1,\ldots,\varphi_n$ and a $\pi$ for each transposition.
\end{proof}

\begin{example}\label{ex:orientations}%
	In \Cref{ex:figura}, $\det M(u,v) > 0$, so the basis $(u,v)$ of $\C^2$ has the canonical complex orientation.
	Indeed, it differs from the canonical basis by complex scalings (produced by the diagonal elements $2$ and $3$ of $M$) and a complex shear (produced by the off-diagonal element $4\im$), with no phase rotations.
	And $(\im u, v)$ gives $\C^2$ a complex orientation differing from the canonical one by a phase $\frac\pi2$, corresponding to a phase rotation of $\C u$ taking $u$ to $\im u$.

	In \Cref{ex:totally real0}, $\arg(\det M(u,v)) = \pi$, so the complex orientation of $(u,v)$ is opposite the canonical one in the circle $\OO_{\C^2}$.
		\SELF{$(-u,v)$ and $(v,u)$ have the canonical one.}
	In \Cref{ex:square}, the complex orientation of the basis $\{v\}$ of $\C$ differs from the canonical one of $\{u\}$ by $\varphi = \frac\pi2$.
	In \Cref{ex:3 em C3 parte0}, that of $(u,v,w)$ differs from the canonical one of $\C^3$ by $\varphi = \arg(\det M) = \frac{3\pi}{4}$.
	In \Cref{ex:C3}, $(u,v)$ gives $V = \Span_\C\{u,v\} \subset \C^3$ a complex orientation, but $V$ has no canonical one for comparison.	

	In \Cref{ex:complex T},
	$T$ is diagonalizable, so it consists only of scalings and phase rotations of the complex lines of its eigenvectors,
	but with scaling factors and phases given by complicated eigenvalues.
	It also decomposes as
	\begin{equation*}
		T = \begin{pmatrix}
			1 & 1-\im \\
			0 & 1
		\end{pmatrix}\!\!
		\begin{pmatrix}
			1 & 0 \\
			-1 & 1
		\end{pmatrix}\!\!
		\begin{pmatrix}
			1 & 1+\im \\
			0 & 1
		\end{pmatrix}\!\!
		\begin{pmatrix}
			1 & 0 \\
			\im & 1
		\end{pmatrix}\!\!
		\begin{pmatrix}
			1 & 2-\im \\
			0 & 1
		\end{pmatrix}\!\!
		\begin{pmatrix}
			1 & 0 \\
			-\frac12 & 1
		\end{pmatrix}\!\! 
		\begin{pmatrix}
			4 e^{-\im \frac{\pi}{6}} & 0 \\
			0 & 1
		\end{pmatrix},
	\end{equation*}	
	with 6 complex shears, a single complex line scaling by $|\det T|=4$ and a single phase rotation by $\varphi = \arg(\det T) = -\frac{\pi}{6}$.
	A more interesting decomposition is
	\begin{equation}\label{eq:T}
		T =
		\begin{pmatrix}
			1 & \im-1 \\
			0 & 1
		\end{pmatrix}
		\begin{pmatrix}
			\cos \frac\pi4 & -\sin \frac\pi4 \\
			\sin \frac\pi4 & \cos \frac\pi4
		\end{pmatrix}
		\begin{pmatrix}
			e^{\im\frac{\pi}{3}} & 0 \\
			0 & e^{-\im\frac{\pi}{2}}
		\end{pmatrix}
		\begin{pmatrix}
			2\sqrt{2} & 0 \\
			0 & \sqrt{2}
		\end{pmatrix},
	\end{equation}
	with a complex shear $e_2 \mapsto e_2 + (\im-1) e_1$ and
	a complex plane rotation by $\frac\pi4$, which do not affect $\det T$;
	phase rotations by $\frac\pi3$ and $-\frac\pi2$,
	rotating $\OO_{\C^2}$ by a total phase $\varphi = -\frac\pi6$;
	and 2 complex line scalings, dilating the 2 real dimensions of $\C e_1$ by $2\sqrt{2}$, those of $\C e_2$ by $\sqrt{2}$, and 4-volumes by $(2\sqrt{2} \cdot \sqrt{2})^2 = 16 = |\det T|^2$.
	In the underlying $\R^4$, 
	the decomposition of $T_\R$ corresponding to \eqref{eq:T} is
	\begin{multline*}
		T_\R = 
		\begin{psmallmatrix}
			1 \,& 0 & -1 & -1 \\[2pt]
			0 \,& 1 & 1 & -1 \\[2pt]
			0 \,& 0 & 1 & 0 \\[2pt]
			0 \,& 0 & 0 & 1
		\end{psmallmatrix}
		\begin{psmallmatrix}
			\cos \frac\pi4 & 0 & -\sin \frac\pi4 & 0 \\
			0 & \cos \frac\pi4 & 0 & -\sin \frac\pi4 \\
			\sin \frac\pi4 & 0 & \cos \frac\pi4 & 0 \\
			0 & \sin \frac\pi4 & 0 & \cos \frac\pi4
		\end{psmallmatrix}
		\cdot \\ \cdot
		\begin{psmallmatrix}
			\cos \frac\pi3 & -\sin \frac\pi3 & 0 & 0 \\
			\sin \frac\pi3 & \cos \frac\pi3 & 0 & 0 \\
			0 & 0 & \cos \frac\pi2 & \sin \frac\pi2 \\
			0 & 0 & -\sin \frac\pi2 & \cos \frac\pi2
		\end{psmallmatrix}
		\begin{psmallmatrix}
			2\sqrt{2} & 0 & 0 & 0 \\
			0 & 2\sqrt{2} & 0 & 0 \\
			0 & 0 & \,\sqrt{2}\, & 0 \\
			0 & 0 & 0 & \,\sqrt{2}
		\end{psmallmatrix},
	\end{multline*}
	with 2 real shears%
		\footnote{$f_3 \mapsto f_3 - f_1 + f_2$ and $f_4 \mapsto f_4 - f_1 - f_2$, for the canonical basis $(f_1,\ldots,f_4)$ of $\R^4$.}
	and 4 real plane rotations (by $\frac\pi4$, $\frac\pi4$, $\frac\pi3$ and $-\frac\pi2$), which do not affect $\det T_\R$; and 4 real line dilations (by $2\sqrt{2}$, $2\sqrt{2}$, $\sqrt{2}$ and $\sqrt{2}$) which dilate 4-volumes by $\det T_\R = 16$.
	Note how these transformations preserve $\OO_{\R^4}$.
\end{example}

\section{Multivectors}\label{sc:Complex blades}

First we review Grassmann algebras \cite{Bourbaki1989,Greub1978,Rosen2019,Shaw1983},
	\CITE{Shaw inclui álgebra complexa, usa decomposable multivector; Rosen usa simple multivector, faz só real}
which in the complex case seem more ``geometric'' than Clifford ones (as we discuss in \Cref{sc:Complex numbers and Geometric Algebra}).

The \emph{Grassmann (exterior) algebra} of a $p$-subspace $V \subset \F^n$
is a graded algebra $\bigwedge V = \bigoplus_{k=0}^p \bigwedge^k V$
of \emph{multivectors},
with an $\F$-bilinear, associative and alternating%
	\footnote{This means $v\wedge v=0$, and so $u \wedge v = -v\wedge u$, for $u,v\in V$.}
\emph{exterior product} $\wedge$.
While $\bigwedge^0 V = \F$, elements of $\bigwedge^k V$ for $k \geq 1$ are sums of \emph{$k$-blades}
$B = v_1\wedge\cdots\wedge v_k$ for $v_1,\ldots,v_k\in V$ (also called \emph{simple} or \emph{decomposable $k$-vectors}).
This decomposition of $B$ is not unique,
but
$B \neq 0 \Leftrightarrow v_1,\ldots,v_k$ are linearly in\-de\-pend\-ent,
in which case $B$ determines
	\OMIT{If $u_1\wedge\cdots\wedge u_k = v_1\wedge\cdots\wedge v_k \neq 0$ then $\Span\{u_1,\ldots,u_k\} = \Span\{v_1,\ldots,v_k\}$, as, for all $j$, $u_j \wedge v_1\wedge\cdots\wedge v_k = u_j \wedge u_1\wedge\cdots\wedge u_k = 0$ implies $u_j \in \Span\{v_1,\ldots,v_k\}$.}
a $k$-subspace $[B] = \Span\{v_1,\ldots,v_k\}$,
	\OMIT{$= \{v\in V:v\wedge B = 0\}$}
oriented%
	\footnote{If $B = u_1 \wedge \cdots \wedge u_k$ is another decomposition, 
	$(u_1,\ldots,u_k)$ has the same orientation since the change-of-basis transformation given by $T v_j = u_j$ has $\det T = 1$, by \eqref{eq:outermorphism}.}
by $(v_1,\ldots,v_k)$.
For another $k$-blade $A \neq 0$, 
$[A] = [B] \Leftrightarrow A = cB$ for $c\in \F$, 
and orientations coincide when $c>0$, by \eqref{eq:outermorphism}.

The inner (Hermitian, if $\F=\C$) product extends to $\bigwedge V$, with%
	\footnote{It differs from Hestenes inner product of Geometric Algebra by a reversion, $A\cdot B = \inner{\tilde{A},B}$.}
$\inner{A,B} = \det\!\big(\inner{u_i,v_j}\big)$ for $A=u_1\wedge\cdots\wedge u_k$ and $B=v_1\wedge\cdots\wedge v_k$.
The norm of $B$ is
\begin{equation}\label{eq:norm}
	\|B\|=\sqrt{\inner{B,B}} = \sqrt{\det G(v_1,\ldots,v_k)}.
\end{equation}
If $[A]$ is $\F$-orthogonal to $[B]$ then $\|A \wedge B\|=\|A\|\|B\|$.
For $u,v \in \F^n$ we have%
	\footnote{Recall that $\gamma_{u,v}$ is the Hermitian angle (\Cref{sc:Determinants}).}
\begin{equation}\label{eq:norm 2-blades}
	\|u \wedge v\| = 
	\begin{cases}
		\|u\|\|v\|\sin\theta_{u,v} &\text{in the real case}, \\
		\|u\|\|v\|\sin\gamma_{u,v} &\text{in the complex case}.
	\end{cases}
\end{equation}

If $\beta = \{v_1,\ldots,v_p\}$ is a basis of $V$,
$\{v_{i_1}\wedge\cdots\wedge v_{i_k}: 1\leq i_1 < \cdots<i_k\leq p\}$ is one of $\bigwedge^k V$. It is orthonormal if so is $\beta$.
So, $\dim \bigwedge^k V = \binom{p}{k}$ and $\dim\bigwedge V = 2^p$.

A linear transformation $T$ of $V = \Span\{v_1,\ldots,v_p\}$ induces another (an \emph{outermorphism}) $\bigwedge^p T$ of $\bigwedge^p V = \Span\{v_1\wedge\cdots\wedge v_p\}$, with
\begin{equation}\label{eq:outermorphism}
	{(\textstyle \bigwedge^p T)}(v_1 \wedge \cdots\wedge v_p) = (T v_1) \wedge\cdots\wedge(T v_p) = (\det T) v_1 \wedge \cdots\wedge v_p.
\end{equation}
This gives a well known (but somewhat abstract) interpretation for $\det T$:
\begin{itemize}
	\item  $|\det T|$ is the factor by which $\bigwedge^p T$ scales $\bigwedge^p V \cong \F$.\label{it:|det| outermorphism}
	
	\item If $\det T \neq 0$, its argument is the angle by which $\bigwedge^p T$ rotates $\bigwedge^p V$. \label{it:arg rotates bigwedge}
\end{itemize}
Note that the scaling factor does not depend on $\F$.
As \Cref{pr:blade norms} will show, the square in \Cref{pr:det scale}\ref{it:complex det scale} is due to how complex blades relate to volumes.

The following alternative definitions for $\OO_V$ are equivalent to \Cref{df:orientation}, by \eqref{eq:outermorphism},
and \ref{it:orientation unit blades} clearly shows $\OO_V$ as a unit circle in $\bigwedge^p V \cong \F$:
\begin{enumerate}[label = (\roman*)]
	\item $\OO_V = \{0\neq B \in \bigwedge^p V\}/\!\sim$, with $B_1 \sim B_2$ if $B_2 = c B_1$ for $c>0$;\label{it:orientation blades}	
	
	\item $\OO_V = \{B \in\bigwedge^p V: \|B\|=1\}$.\label{it:orientation unit blades}
\end{enumerate}

The Grassmann algebras of $\C^n$ and $(\C^n)_\R = \R^{2n}$ are different, as an exterior product is  $\C$-bilinear and the other is $\R$-bilinear.
And $\bigwedge (\C^n)_\R \neq (\bigwedge \C^n)_\R$, as their $\R$-dimensions are $2^{2n}$ and $2^{n+1}$, respectively.
	\OMIT{$= 2 \dim_\C \bigwedge \C^n$}

\begin{notation}
	In the complex case, we use $\rwedge$ for the $\R$-bilinear exterior product of $\bigwedge (\C^n)_\R$, to distinguish it from the $\C$-bilinear exterior product $\cwedge$ of $\bigwedge \C^n$.
\end{notation}

\subsection{Blades -- 1st interpretation}\label{sc:blades}

A blade $B = v_1 \wedge \cdots \wedge v_p  \neq 0$ conveys 3 data pieces:
a number $c = \|B\|$, 
associated to a $p$-subspace $[B] =\Span\{v_1,\ldots,v_p\}$,
oriented by the basis $(v_1,\ldots,v_p)$.
And such data determines $B$ uniquely, as given an orthonormal basis $(u_1,\ldots,u_p)$ of same orientation for the subspace, we have $B = c u_1 \wedge\cdots\wedge u_p$.

The meaning of $\|B\|$ depends on that of $B$:
for a 1-blade (a vector), it can be a length, velocity, strength of a force, etc.;
a 2-blade can be a magnetic field,
	\SELF{To be precise, at each point it is a 2-blade in the exterior algebra of the dual space $(\R^3)^*$.}
whose intensity is its norm; etc.
A common interpretation in the real case is that $\|B\|$ is a volume,
but the complex case is different (see also \Cref{pr:complex real concepts}):

\begin{theorem}\label{pr:blade norms}
	Let $B = v_1 \wedge \cdots \wedge v_p \in \bigwedge^p \F^n$ for $v_1,\ldots,v_p \in \F^n$.
	\begin{enumerate}[label = (\roman*)]
		\item $\|B\|=\VV_p(v_1,\ldots,v_p)$, in the real case.\label{it:real blade}
		
		\item $\|B\| = \sqrt{\VV_{2p}(v_1,\im v_1,\ldots,v_p,\im v_p)}$, in the complex case.\label{it:complex blade}
	\end{enumerate}
\end{theorem}
\begin{proof}
	Follows from \eqref{eq:norm} and \Cref{pr:Gramian}.
\end{proof}

Due to \ref{it:real blade}, in the real case
$B$ is usually represented by $\PP(v_1,\ldots,v_p) \subset [B]$, oriented by $(v_1,\ldots,v_p)$.
As its decomposition is not unique, $B$ is actually an equivalence class of all oriented parallelotopes (or even regions) of same $p$-volume $\|B\|$, in the same subspace, with same orientation, as any of them carries all data needed to determine $B$.

By \ref{it:complex blade}, neither $\PP(v_1,\ldots,v_p)$ nor $\PP(v_1,\im v_1,\ldots,v_p,\im v_p)$ can represent $B$ adequately in the complex case%
	\footnote{It is not even possible to redefine $\|B\|$ to match the volumes of these parallelotopes,
	as we would still have $\|(\lambda v_1)\wedge v_2\wedge\cdots\wedge v_p\| = |\lambda|\cdot \|v_1\wedge v_2\wedge\cdots\wedge v_p\|$ for $\lambda \in \C$, but  
	$\VV_p(\lambda v_1,\ldots,v_p) \neq |\lambda| \cdot \VV_p(v_1,\ldots,v_p)$,
	as seen, and $\VV_{2p}(\lambda v_1,\im \lambda  v_1,\ldots,v_p,\im v_p) = |\lambda|^2 \cdot \VV_{2p}(v_1,\im v_1,\ldots,v_p,\im v_p)$.}.
In \Cref{sc:2nd interpretation} we will obtain a representation, after linking $B$ to $v_1\rwedge \cdots \rwedge v_p \in \bigwedge(\C^n)_\R$ in \Cref{sc:Holomorphy}.
	\OMIT{\ref{pr:blades complex real}}
But first we have:

\begin{corollary}\label{pr:norm cwedge}
	$\|v_1\cwedge \cdots \cwedge v_p\|^2 = \|v_1 \rwedge \im v_1 \rwedge \cdots \rwedge v_p \rwedge \im v_p\|$ for $v_1,\ldots, v_p \in \C^n$.
\end{corollary}
\begin{proof}
	By \Cref{pr:blade norms}, both sides give $\VV_{2p}(v_1,\im v_1,\ldots,v_p,\im v_p)$.
\end{proof}

Below, $(e_1,\ldots,e_n)$ and $(f_1,\ldots,f_{2n})$ are the canonical bases of $\C^n$ and of the underlying $\R^{2n}$, so $f_1=e_1$, $f_2=\im e_1$, \ldots, $f_{2n-1}=e_n$, $f_{2n}=\im e_n$.
We use the notation $e_{12} = e_1 \wedge e_2$, $f_{134} = f_1 \rwedge f_3 \rwedge f_4$, etc.
Note how calculations in $\bigwedge \C^n$ are simpler than in $\bigwedge \R^{2n}$, as $\dim_\R \bigwedge\R^{2n}  = (\dim_\C \bigwedge \C^n)^2$.
	\OMIT{$= 2^{2n}$}

\begin{example}\label{ex:figura3}%
		\OMIT{$u,v \in\C^2$, Figura, pure real, \Cref{ex:figura,ex:orientations,ex:figura3,ex:index,ex:figura6}}%
	In \Cref{ex:figura},
	$u = 2e_1$ and $v = 4\im e_1 + 3 e_2$,  so $u \cwedge v = 6 e_{12}$.
	In the underlying $\R^4$,
	$u =2f_1$, $\im u = 2f_2$, $v =4f_2+3f_3$ and $\im v =-4f_1+3f_4$, 
	so 
	$u \rwedge v = 8f_{12} + 6f_{13}$
	and
	$\im u \rwedge \im v = 8f_{12} + 6f_{24}$,
	while
	$u \rwedge \im u \rwedge v \rwedge \im v = 36f_{1234}$.
	Thus 
	$\|u \rwedge v\| = \|\im u \rwedge \im v\| = \sqrt{8^2+6^2} = 10 = \VV_2(u,v) \neq \|u \cwedge v\| = 6$
	and 
	$\|u \cwedge v\|^2 = \|u \rwedge \im u \rwedge v \rwedge \im v\| = 36 = \VV_4(u,\im u,v,\im v)$.
\end{example}

\begin{example}\label{ex:totally real}%
		\OMIT{$u,v\in\C^2$, tot real, \Cref{ex:totally real0,ex:orientations,ex:totally real}}%
	In \Cref{ex:totally real0}, $u=5\im e_1$ and $v=4\im e_2$, so $u\wedge v = -20 e_{12}$.
	In $\R^4$, $u=5 f_2$, $\im u = -5 f_1$, $v= 4f_4$ and $\im v = -4 f_3$,
	so 
	$u \rwedge v = 20 f_{24}$
	and
	$u \rwedge \im u \rwedge v  \rwedge \im v = 400 f_{1234}$.
	In this case,
	$\|u\wedge v\| = \|u\rwedge v\| = 20 = \VV_2(u,v)$ and $\|u \wedge v\|^2 = \|u \rwedge \im u \rwedge v  \rwedge \im v\| = 400 = \VV_4(u,\im u,v,\im v)$.
\end{example}

\begin{example}\label{ex:square2}%
		\OMIT{$u,v \in \C$, square $S$, holom, só $u$ tot real, \Cref{ex:square,ex:orientations,ex:square2} (\ref{ex:square4} removido)}%
	In \Cref{ex:square},
	with $u$ and $v$ as vectors,
	$u \wedge v = \im (u \wedge u) = 0$.
	In $\R^2$,
	$u = 2 f_1$, $v = 2 f_2$, and
	$u \rwedge v = 4 f_{12}$ is represented by $S$, of area $\|u \rwedge v\| = 4$, while $u \rwedge \im u \rwedge v \rwedge \im v = u \rwedge \im u \rwedge \im u \rwedge (-u) = 0$.
\end{example}

\begin{example}\label{ex:3 em C3 parte1}%
		\OMIT{$u,v,w\in \C^3$, \Cref{ex:3 em C3 parte0,ex:orientations,ex:3 em C3 parte1,ex:index}}
	In \Cref{ex:3 em C3 parte0},
	$u = e_1 - \im e_3$,
	$v = (1+\im)e_1 + e_2$ and 
	$w = (1+2\im)e_1 - e_2 + \im e_3$,
	so $u \wedge v \wedge w = (3\im-3) e_{123}$.
	In $\R^6$,
	\begin{align*}
		u &= f_1 - f_6, & \im u &= f_2 + f_5, \\
		v &= f_1 + f_2 + f_3, & \im v &= -f_1 + f_2 + f_4, \\
		w &= f_1 + 2 f_2 - f_3 + f_6, & \im w &= -2 f_1 + f_2 - f_4 - f_5,
	\end{align*}
	so
	$u \rwedge v \rwedge w = -3f_{123} + 3f_{136} + 3 f_{236}$
	and 
	$u \rwedge \im u \rwedge v \rwedge \im v \rwedge w \rwedge \im w = 18 f_{123456}$.
	Thus
	$\|u \rwedge v \rwedge w\| = \sqrt{27} = \VV_2(u,v,w)$
	and $\|u \rwedge \im u \rwedge v \rwedge \im v \rwedge w \rwedge \im w\| = \|u \wedge v \wedge w\|^2 = 18 = \VV_6(u,\im u,v,\im v,w,\im w)$.
\end{example}

\begin{example}\label{ex:C3 nova}%
		\OMIT{$u,v\in\C^3$, G, \Cref{ex:orientations,ex:C3 nova}}%
	In \Cref{ex:C3}, 
	$u = (1+2\im) e_1 + 3\im e_3$ and $v = 2e_1 + \im e_2 + (3+\im) e_3$,
	so
	$u \cwedge v = (\im-2) e_{12} + (1+\im)e_{13} + 3e_{23}$.
	Lengthy calculations, as above,
	give $\|u \rwedge v\| = \sqrt{185} = \VV_2(u,v)$ 
	and 
	$\|u \rwedge \im u \rwedge v \rwedge \im v\| = \|u\wedge v\|^2 = 16 = \VV_4(u,\im u,v,\im v)$.
\end{example}

\begin{example}%
		\OMIT{$T$ in $\C^2$, sem 0s, \Cref{ex:complex T,ex:orientations}, (\ref{ex:blade norms} removido)}
	In \Cref{ex:complex T}, 
	$(\bigwedge^2 T)(e_{12}) = (T e_1) \wedge (T e_2) = [(\im - \sqrt{3}) e_1 + (1 +\im \sqrt{3}) e_2] \wedge [(1+2\im)e_1 - \im e_2]
	= 4 e^{-\im \frac\pi6} e_{12} = (\det T) e_{12}$,
	and so $\bigwedge^2 T$ dilates $\bigwedge^2 \C^2 = \Span_\C\{e_{12}\} \cong \C$ by $4 = |\det T|$ (areas dilate by $16$) and rotates it by $\varphi = -\frac\pi6 = \arg(\det T)$.
	As one can check,
	$(\bigwedge^4 T_\R)(f_{1234}) = 16 f_{1234}$,
	so $\bigwedge^4 T_\R$ dilates $\bigwedge^4 \R^4 = \Span_\R\{f_{1234}\} \cong \R$ by $16 = \det T_\R$ and preserves orientation.
\end{example}

\subsection{Volumetric Pythagorean theorem}\label{sc:Pythagorean}

Now we give a simpler proof for a volumetric Pythagorean theorem  (\Cref{fig:Pythagorean}) from \cite{Mandolesi_Pythagorean}.
Other results from that article can be proven similarly.

\begin{figure}
	\centering
	\begin{subfigure}[b]{0.47\textwidth}
		\centering
		\includegraphics[width=.93\textwidth]{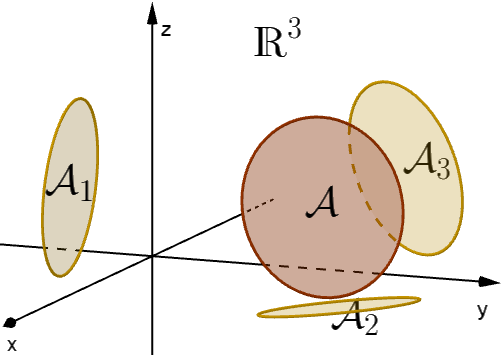}
		\caption{$\AA^2= \AA_1^2+\AA_2^2+\AA_3^2$ for a plane area $\AA$ with orthogonal projections $\AA_1,\AA_2,\AA_3$ on the coordinate planes $xz$, $xy$, $yz$ of $\R^3$.}
		\label{fig:projecao_planos}
	\end{subfigure} 
	\hspace{1mm}
	\begin{subfigure}[b]{0.50\textwidth}
		\centering
		\includegraphics[width=.75\textwidth]{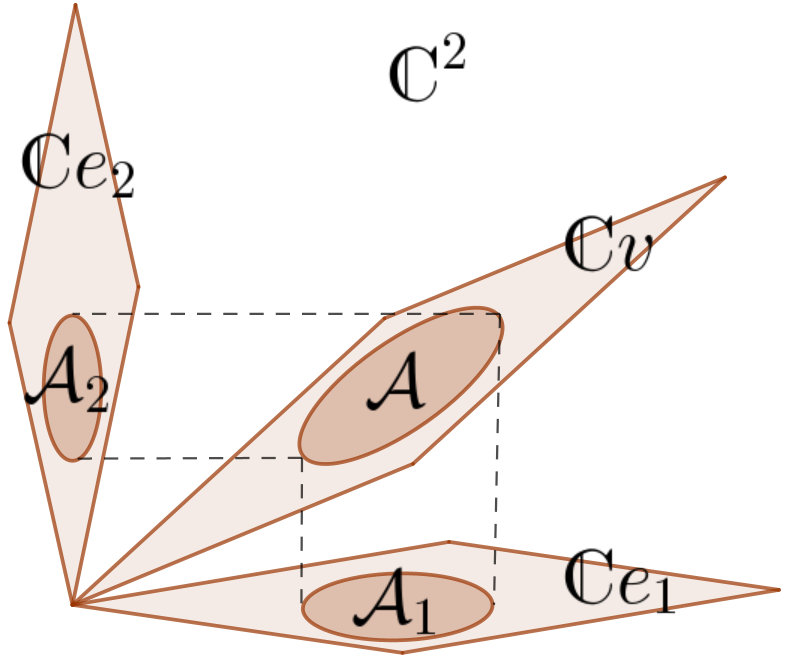}
		\caption{$\AA=\AA_1+\AA_2$ for an area $\AA$ in $\C v$ with orthogonal projections $\AA_1,\AA_2$ on the complex lines of an orthogonal basis $(e_1,e_2)$ of $\C^2$.}
		\label{fig:complex_lines}
	\end{subfigure}
	\caption{Real and complex volumetric Pythagorean theorems}
	\label{fig:Pythagorean}
\end{figure}

Let $(e_1,\ldots,e_n)$ be an orthonormal basis of $\F^n$.
The associated orthonormal basis of $\bigwedge^p \F^n$ is
$\{e_I = e_{i_1} \wedge \cdots \wedge e_{i_p}:I=(i_1,\ldots,i_p), 1\leq i_1<\cdots<i_p\leq n\}$.
Each $[e_I] = \Span\{e_{i_1},\ldots,e_{i_p}\}$ is a \emph{coordinate $p$-subspace}.

\begin{theorem}
	If a subset of a $p$-subspace $V\subset \F^n$ has $k$-volume $\VV$ (with $k=p$ if $\F=\R$, or $k=2p$ if $\F=\C$), and its orthogonal projection on $[e_I]$ has $\VV_I$, then
	$\VV^2 = \sum_I \VV_I^2$ if $\F=\R$, or
	$\VV = \sum_I \VV_I$ if $\F=\C$.
\end{theorem}
\begin{proof}
	Given a basis $\{v_1,\ldots,v_p\}$ of $V$,
	it is enough to prove it for 
	$\PP(v_1,\ldots,v_p)$  if $\F=\R$,
	or $\PP(v_1,\im v_1,\ldots,v_p,\im v_p)$ if $\F=\C$.
	The orthogonal projection of $B = v_1\wedge \cdots \wedge v_p$ on $\Span\{e_I\}$ is $B_I = P_{[e_I]} v_1 \wedge \cdots \wedge P_{[e_I]} v_p$,
	so the result follows from $\|B\|^2 = \sum_I \|B_I\|^2$ and \Cref{pr:blade norms}.
\end{proof}

In the complex case, one might expect orientations to play a role in how the $\VV_I$'s combine to form $\VV$.
This does not happen because the phase of each orthogonal projection $B_I$ is aligned with that of $B$, in the sense that $\inner{B_I,B} = \|B_I\|^2 \geq 0$.

\begin{example}
	In \Cref{ex:C3 nova}, $\PP = \PP(u,\im u,v,\im v)$ has  4-volume $\VV = \|B\|^2 = 16$, where $B= u \wedge v$.
	Its orthogonal projection on $[e_{12}] = \Span_\C\{e_1,e_2\}$ is the parallelotope spanned (writing $P=P_{[e_{12}]}$) by $Pu = (1+2\im)e_1$, $Pv = 2e_1 + \im e_2$, $\im Pu$ and $\im Pv$, with $\VV_{12} = \|B_{12}\|^2 = 5$ for $B_{12} = P(B) = Pu \wedge Pv = (\im-2)e_{12}$.
	Likewise, $V_{13} = \|B_{13}\|^2 = 2$ with $B_{13} = (1+\im) e_{13}$, and $V_{23} = \|B_{23}\|^2 = 9$ with $B_{23} = 3 e_{23}$,
	so that $\VV_{12} + \VV_{13} + \VV_{23} = 5+2+9 = \VV$.
	
	In the underlying $\R^6$, $\PP$ is represented by $A = u \rwedge \im u \rwedge v \rwedge \im v$, with
	\begin{align*}
		u &= f_1 + 2f_2 + 3 f_6, & \im u &= f_2 - 2f_1 - 3 f_5, \\
		v &= 2f_1 + f_4 + 3 f_5 + f_6, & \im v &= 2f_2 - f_3 + 3 f_6 - f_5,
	\end{align*}
	and $\VV = \|A\| = 16$.
	Its orthogonal projection on $[f_{1234}] = \Span_\R\{f_1,\ldots,f_4\}$ (the underlying real space of $[e_{12}]$) is represented by $A_{1234} = (f_1 + 2f_2) \rwedge (f_2 - 2f_1) \rwedge (2f_1 + f_4) \rwedge (2f_2 - f_3) = 5 f_{1234}$, so $V_{1234} = \|A_{1234}\| = 5$ (same as the above $\VV_{12}$).
	Similar calculations for all 15 coordinate 4-subspaces confirm that the sum of the squared projected volumes equals $\VV^2$.
\end{example}

Note how the relation between volumes is simpler in the complex case not only because they are not squared, but also because there are less coordinate subspaces.

\section{Holomorphy}\label{sc:Holomorphy}

As in \Cref{rm:iP P non orthog}, $\|v_1 \rwedge \im v_1 \rwedge \cdots \rwedge v_p \rwedge \im v_p\| \neq \|v_1\rwedge \cdots \rwedge v_p\|^2$, 
in general,
and so $\|v_1\cwedge \cdots \cwedge v_p\| \neq \|v_1\rwedge \cdots \rwedge v_p\|$ by \Cref{pr:norm cwedge},
because $\PP(v_1,\ldots,v_p)$ and
$\PP(\im v_1,\ldots, \im v_p)$ are not orthogonal.
To relate these norms we must see how the complex structure $\im$ rotates $V = \Span_\R\{v_1,\ldots, v_p\}$ to $\im V = \Span_\R\{\im v_1,\ldots, \im v_p\}$.

\begin{definition}
	A real subspace $V \subset (\C^n)_\R$ is 
	\emph{holomorphic} if $V = \im V $;
	it is \emph{purely real}%
		\footnote{The terminology is from \cite{Goldman1999}. Some authors use totally real to mean purely real.}
		\CITE{Rosenfeld1997 uses the term \emph{antiholomorphic} for totally real planes.}
	if $V \cap \im V = \{0\}$;
	and 
	\emph{totally real} if $V \perp_\R \im V$.
\end{definition}
\begin{figure}
	\centering
	\begin{subfigure}[b]{0.32\textwidth}
		\includegraphics[width=\textwidth]{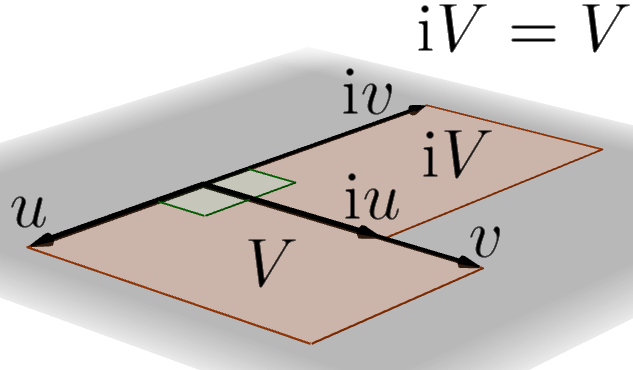}
		\caption{Holomorphic: if $v \in \C u$ then $\im$ rotates $V$ by $90^\circ$ inside itself, and $\im V = V = \C u$.}
		\label{fig:holomorfico}
	\end{subfigure} 
	\hspace{.15mm}
	\begin{subfigure}[b]{0.33\textwidth}
		\includegraphics[width=\textwidth]{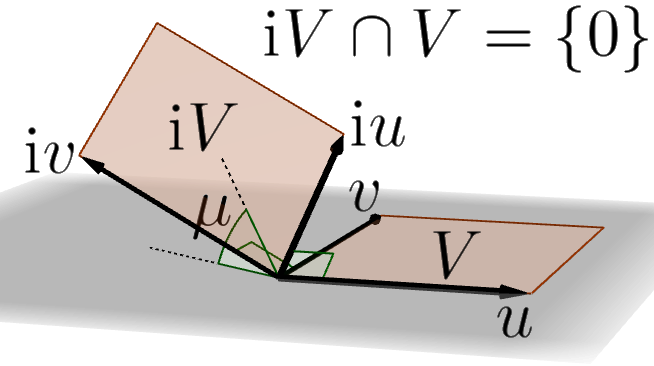}
		\caption{Purely real: $\im$ ``lifts'' $V$ by a Kähler angle $\mu$ while rotating its vectors $90^\circ$.}
		\label{fig:pure real}
	\end{subfigure}
	\hspace{.15mm}
	\begin{subfigure}[b]{0.31\textwidth}
		\includegraphics[width=\textwidth]{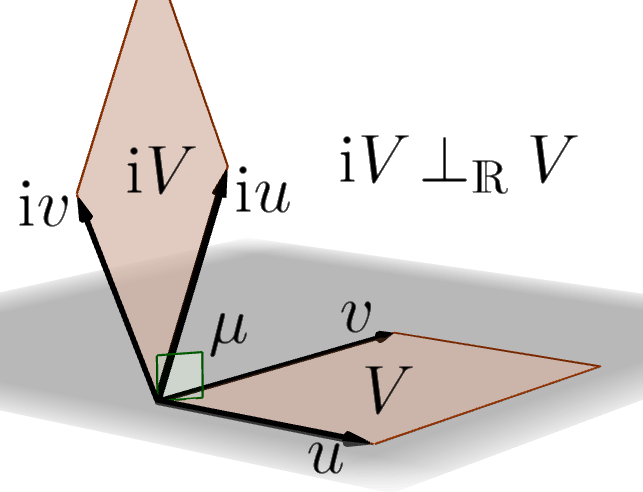}
		\caption{Totally real: $\im$ rotates $V$ $\R$-orthogonally \wrt itself.}
		\label{fig:total real}
	\end{subfigure}
	\caption{Degrees of holomorphy of $V=\Span_\R\{u,v\} \subset (\C^n)_\R$, depending on how the complex structure $\im$ rotates it. To properly represent the relation between the real planes $V$ and $\im V$, we only show the part of $\im V$ inside the parallelogram $\PP(\im u,\im v)$.}
	\label{fig:holomorphy}
\end{figure}

See \Cref{fig:holomorphy}. The following interpretation of these concepts is well known, but we include a simple proof for the sake of completeness:
\begin{proposition}
	Let $V \subset (\C^n)_\R$ be a real subspace.
\begin{enumerate}[label = (\roman*)]
	\item $V$ is holomorphic $\Leftrightarrow$ $V$ is the underlying real space of a complex subspace.\label{it:holomorphic}
	
	\item $V$ is purely real $\Leftrightarrow V$ does not contain a holomorphic subspace $\Leftrightarrow$ any $\R$-basis $(v_1,\ldots,v_p)$ of $V$ is $\C$-linearly independent. \label{it:purely real}
	\SELF{$\Leftrightarrow$ the elements of an $\R$-basis of $V$ are $\C$-linearly independent.}
	
	\item $V$ is totally real $\Leftrightarrow \inner{\cdot,\cdot} = \inner{\cdot,\cdot}_\R$ on $V$. \label{it:totally real}
	\CITE{Goldman1999 p.35}
\end{enumerate}
\end{proposition}
\begin{proof}
	\ref{it:holomorphic} $V = \im V \Leftrightarrow V$ is $\C$-closed.
	\ref{it:purely real} No subspace of $V$ is $\C$-closed $\Leftrightarrow V \cap \im V = \{0\} \Leftrightarrow 2\dim_\C \left(\Span_\C\{v_1,\ldots,v_p\}\right) = \dim_\R (V + \im V) = \dim_\R V + \dim_\R \im V = 2p$.
	\ref{it:totally real} Follows from  \eqref{eq:inner}.
\end{proof}

The lack of holomorphy of a real plane $V$ is usually measured by a
\emph{Kähler or holomorphy angle} 
\cite{Goldman1999,Rosenfeld1997,Scharnhorst2001}
	\CITE{Goldman p.17,36 holomorphy angle $\mu$ \\ Rosenfeld p.182 holom $\alpha$ \\ Scharnhorst Kahler $\Theta_K$}
given, for an $\R$-orthogonal basis $\{u,v\}$, 
by $\mu_V = \cos^{-1} \frac{|\inner{\im u, v}_\R|}{\|u\|\|v\|} = \gamma_{u,v}$.
\OMIT{\eqref{eq:inner}}	
It equals both principal angles (see \Cref{Angles between subspaces}) of $V$ and $\im V$ (\Cref{fig:pure real}).
To generalize it for higher dimensional subspaces, we use the following result:

\begin{proposition}\label{pr:Kahler}
	If $\theta_1\leq\cdots\leq \theta_p$ are the principal angles a real $p$-subspace $V \subset (\C^n)_\R$ makes with $\im V$ then:
	\begin{enumerate}[label = (\roman*)]
		\item $\theta_{2j-1} = \theta_{2j}$ for $1\leq j \leq \lfloor \frac p2 \rfloor$, and $\theta_p = \frac\pi2$ if $p$ is odd. \label{it:repeated thetas}
		
		\item $V$ and $\im V$ have principal bases $(v_1,\ldots,v_p)$ and $(-\im v_2, \im v_1, \ldots, -\im v_p, \im v_{p-1})$ if $p$ is even, or $(-\im v_2, \im v_1, \ldots, -\im v_{p-1}, \im v_{p-2}, \im v_p)$ if $p$ is odd. \label{it:principal bases}
	\end{enumerate}
\end{proposition}
\begin{proof}
	If $p=1$ or $\theta_1 = \frac\pi2$ then $V \perp_\R \im V$ and the result follows.
	Assuming otherwise,
	choose principal vectors $v_1 \in V$ and $\im v \in \im V$ with $\theta_{v_1,\im v} = \theta_1 \neq \frac\pi2$.
	As $P_V(\im v) = v_1 \cos \theta_1 \neq 0$ and $\im$ is a $\frac\pi2$ rotation, we have
	$v \perp_\R v_1$, 
		\OMIT{$\im v\perp_\R v \in V$, so $P_V(\im v) \perp_\R v$}
	$\im v \perp_\R \im v_1$
	and	$\theta_{-v,\im v_1} = \theta_{v_1,\im v} = \theta_1$.
	By \eqref{eq:theta_j}, $\theta_2 = \theta_1$ and we can take
	$v_2 = -v$ and $\im v_1$ as its principal vectors. 
	As $U = \Span_\R\{v_1,v_2\}$ and $\im U = \Span_\R\{-\im v_2, \im v_1\}$ are spanned by associated principal vectors,
	$V' = U^\perp \cap V \perp_\R \im U$ and $\im V' = (\im U)^\perp \cap\,\im V \perp_\R U$.
	The result follows via induction: if it holds for $V'$, it holds for $V = U \oplus V'$.
\end{proof}

\begin{definition}
	The \emph{Kähler angles}%
		\footnote{The \emph{multiple Kähler angle} of \cite{Tasaki2001} is formed by the $\mu_j$'s but excluding $\mu_h$ if $p$ is odd, and using $V^\perp$ if $n < p \leq 2n$, so it has slightly more complicated properties.}
	of $V$ are $\mu_j = \theta_{2j-1}$ for $1\leq j \leq h = \lceil \frac p2 \rceil$,
	and the \emph{holomorphy index} of $V$ is 
	$\rho_V = \prod_{j=1}^h \sin \mu_j$.
		\FUTURE{Definir ângulo de holomorfia generalizado $\mu = \sin^{-1} \rho_V$?}
\end{definition}

As $\mu_1\leq\cdots\leq\mu_h$ are the principal angles of $V$ and $\im V$, minus the repetition of \ref{it:repeated thetas}, we have that $V$ is:
\begin{itemize}
	\item holomorphic $\Leftrightarrow \mu_j = 0 \ \forall j$;
	\item purely real $\Leftrightarrow \mu_j \neq 0 \ \forall j$;
	\item totally real $\Leftrightarrow \mu_j =\frac\pi2 \ \forall j$.
\end{itemize}
Therefore,
\begin{itemize}
	\item $\rho_V = 0 \Leftrightarrow V$ contains a holomorphic subspace;
	\item $\rho_V \neq 0 \Leftrightarrow V$ is purelly real;
	\item $\rho_V = 1 \Leftrightarrow V$ is totally real.
\end{itemize}
By \ref{it:principal bases}, $V = V_1 \oplus \cdots \oplus V_h$ for $V_j = \Span_\R\{v_{2j-1},v_{2j}\}$ ($V_h = \R v_p$ if $p$ is odd), with $V_j \perp_\R V_k$ and $V_j \perp_\R \im V_k$ for $j \neq k$,
and $\mu_j$ is the usual Kähler angle of $V_j$ ($\mu_h = \frac\pi2$ if $p$ is odd).
	\SELF{$\mu_j$ is also isocliny angle of $V_j$ \wrt $\im V$ (even $j=h$), and of $\C v_{2j-1}$ \wrt $\C v_{2j} \oplus (V \cap V_j^\perp) \oplus \im(V \cap V_j^\perp)$}

In \Cref{Angles between subspaces} we study and interpret the following angle, relating it to principal angles and to the exterior product of blades:

\begin{definition}
	The \emph{disjointness angle} of $V,W \subset \F^n$ is $\Upsilon_{V,W} = \sin^{-1} \frac{\|P_{W^\perp} B\|}{\|B\|}$, where $B$ is a blade with $[B]=V$ and $P_{W^\perp}$ is the orthogonal projection on $W^\perp$.
\end{definition}

For a real subspace $V \subset (\C^n)_\R$ and $W = \im V$, the angle is linked to the holomorphy index $\rho_V$, which then relates real and complex blades:

\begin{proposition}\label{pr:blades complex real}
	Let $v_1,\ldots,v_p \in \C^n$ and $V=\Span_\R \{v_1,\ldots,v_p\}$. 
	\begin{enumerate}[label = (\roman*)]
		\item $\rho_V^2 = \sin \Upsilon_{V,\im V}$.\label{it:rho Upsilon}
			\SELF{angle computed in $(\C^n)_\R$.}
		
		\item $\|v_1 \cwedge \cdots \cwedge v_p\| = \rho_V \cdot \|v_1 \rwedge \cdots \rwedge v_p\|$. \label{it:blade complex real}
	\end{enumerate}
\end{proposition}
\begin{proof}
	\ref{it:rho Upsilon} Follows from Propositions \ref{pr:upsilon}\ref{it:Upsilon sin} and \ref{pr:Kahler}\ref{it:repeated thetas}.
	\ref{it:blade complex real} Follows from \Cref{pr:norm cwedge}, as \Cref{pr:upsilon}\ref{it:norm wedge Upsilon}%
		\OMIT{and \ref{it:rho Upsilon}} 
	gives
	$\|(v_1 \rwedge \cdots \rwedge v_p) \rwedge (\im v_1 \rwedge \cdots \rwedge \im v_p)\|
	=  \|v_1 \rwedge \cdots \rwedge v_p\| \|\im v_1 \rwedge \cdots \rwedge \im v_p\| \sin \Upsilon_{V,\im V} 
	= \|v_1 \rwedge \cdots \rwedge v_p\|^2 \cdot \rho_V^2$.%
		\SELF{$\VV(v_1,\im v_1,\ldots,v_p,\im v_p) = \rho_V^2 \cdot \VV(v_1,\ldots,v_p)^2$}	
\end{proof}

Other real and complex concepts are also related via $\rho_V$:
e.g., \ref{it:blade complex real} and \eqref{eq:norm 2-blades} give $\sin \gamma_{u,v} = \sin \mu_V \cdot \sin \theta_{u,v}$ \cite[eq.\,2.6]{Goldman1999} for $u,v\in\C^n$ and $V= \Span_\R\{u,v\}$, with the Kähler angle linking Hermitian and Euclidean ones.

\begin{example}\label{ex:index}
	\Cref{tab:exemplos} has results for $V = \Span_\R\{u,v\}$ in various examples.
	Note how nonlinear are the relations between $\rho_V$, $\mu_V$ (or $\mu_j$) and $\Upsilon_{V, \im V}$.
		\SELF{so how far from holomorphic $V$ is depends on how one poses the question}		
	In \Cref{ex:figura3},
	$\mu_V$ is
	the angle in \Cref{fig:exemplo1} between $\im u$ and $V$, between $v$ and $\C u$, and between $\im v$ (not shown) and both $V$ and $\C u$.	
	In \Cref{ex:totally real},
	$\|u \cwedge v\| = \|u \rwedge v\|$ as $V$ is totally real ($\im u$ and $\im v$ are $\R$-orthogonal to it).
	In \Cref{ex:square2}, 
	$u \wedge v = 0$ despite $u \rwedge v \neq 0$ because $V = (\C)_\R$ is holomorphic.
	In \Cref{ex:3 em C3 parte1}, a calculation gives principal angles $(\cong)\, 55^\circ$, $55^\circ$ and $90^\circ$ for $V=\Span_\R\{u,v,w\}$ and $\im V$, 
	so the Kähler angles are $\mu_1 \cong 55^\circ$ and $\mu_2 = 90^\circ$,
	and \Cref{pr:upsilon}\ref{it:Upsilon sin} also gives $\Upsilon_{V,\im V} \cong 42^\circ$, which reflects the total effect of the principal angles on the contraction of 3-volumes orthogonally projected from $\im V$ to $V^\perp$.
\end{example}

\begin{table}[t]
	\centering
	\addtolength{\leftskip} {-1.3cm}
	\renewcommand{\arraystretch}{1}
	\begin{tabular}{cccccc}
		\toprule
		Example & \ref{ex:square}/\ref{ex:square2} & \ref{ex:C3}/\ref{ex:C3 nova} & \ref{ex:figura}/\ref{ex:figura3} & \ref{ex:3 em C3 parte0}/\ref{ex:3 em C3 parte1} & \ref{ex:totally real0}/\ref{ex:totally real}  \\
		\cmidrule(lr){1-1} \cmidrule(lr){2-2} \cmidrule(lr){3-3} \cmidrule(lr){4-4} \cmidrule(lr){5-5} \cmidrule(lr){6-6} 
		$\|u\wedge v\| \ (=|\det M(u,v)| \text{ if } p =n)$ & 0 & 4 & 6 & $\sqrt{18}$ & 20
		\\[3pt]	
		$\|u\rwedge v\| = \|\im u\rwedge \im v\| = \VV(u,v)$ & 4 & $\sqrt{185}$ & 10 & $\sqrt{27}$ & 20
		\\[3pt]	
		$\|u\rwedge \im u \rwedge v \rwedge \im v\| = \|u\wedge v\|^2 = \det G(u,v)$ & 0 & 16 & 36 & 18 & 400
		\\[3pt]	
		$\rho_V = \frac{\|u \cwedge v\|}{\|u \rwedge v\|} = \prod_{j=1}^h \sin \mu_j$ &  0 & 0.294  & 0.6 & 0.816 &  1
		\\[3pt]	
		$\mu_V = \sin^{-1} \rho_V$, or $\mu_j = \theta_{2j-1}$ & $0^\circ$ & $17^\circ$ & $37^\circ$ & $55^\circ, 90^\circ$ & $90^\circ$
		\\[3pt]	
		$\Upsilon_{V, \im V} = \sin^{-1} \frac{\|u \rwedge \im u \rwedge v \rwedge \im v\|}{\|u \rwedge v\| \|\im u \rwedge \im v\|} = \sin^{-1} \rho_V^2$ & $0^\circ$ & $5^\circ$ & $21^\circ$ & $42^\circ$ & $90^\circ$
		\\[3pt]		
		\bottomrule
	\end{tabular}
	\caption{Results for $V = \Span_\R\{u,v\}$ (in example \ref{ex:3 em C3 parte0}/\ref{ex:3 em C3 parte1}, include $w$ everywhere)}
	\label{tab:exemplos}
\end{table}

\section{The 2nd interpretation}\label{sc:2nd interpretation}

Complex blades and determinants can now have simpler interpretations in terms of fractions of volumes, without the extra $\im v_j$'s and squares of before:
\OMIT{\Cref{pr:det vol,pr:Gramian,pr:blade norms}}

\begin{theorem}\label{pr:complex real concepts}
	Let $v_1,\ldots,v_p \in \C^n$ and $V=\Span_\R \{v_1,\ldots,v_p\}$. 
	\begin{enumerate}[label = (\roman*)]
		\item $\|v_1 \cwedge \cdots \cwedge v_p\| = \rho_V \cdot \VV_p(v_1,\ldots,v_p)$. \label{it:blades angle}
		\item $|\det M(v_1,\ldots,v_n)| = \rho_V \cdot \VV_n(v_1,\ldots,v_n)$, for $p=n$.\label{it:det angle}
			\SELF{$= \|v_1 \cwedge \cdots \cwedge v_p\|$ \\ 
				$= \det M_\R(v_1,\im v_1,\ldots,v_n,\im v_n)$}
	\end{enumerate}
\end{theorem}
\begin{proof}
	\ref{it:blades angle} Follows from Propositions \ref{pr:blades complex real}\ref{it:blade complex real} and \ref{pr:blade norms}\ref{it:real blade}.
	\ref{it:det angle} Follows from \ref{it:blades angle}, \eqref{eq:norm} and $\det G = |\det M|^2$.
		\SELF{$\det G(v_1,\ldots,v_p) = \rho_V^2 \cdot \VV_p(v_1,\ldots,v_p)^2$}
\end{proof}

As $v_1 \cwedge \cdots \cwedge v_p$ can be decomposed in terms of other vectors in the same complex subspace, but which $\R$-span a different $V$, any change in the holomorphy index must be compensated by another in the volume.
If $(v_1,\ldots,v_p)$ is an $\R$-orthonormal basis of $V$, the index can be computed by $\rho_V  = \|v_1 \cwedge \cdots \cwedge v_p\|$, or by $\rho_V = |\det M(v_1,\ldots,v_n)|$ if $p=n$.

While determinants of linear transformations of $\R^n$ describe only the scaling of $n$-volumes, those of $\C^n$ they give information about $2n$-volumes (\Cref{pr:det scale}) and also $n$-volumes:

\begin{corollary}\label{pr:T holomorphy}
	For a $\C$-linear $T:\C^n\rightarrow\C^n$
	and a measurable subset $E$ of a real $n$-subspace $V \subset (\C^n)_\R$
	we have
	$\rho_{T(V)} \cdot \VV_n(T(E)) = |\det T| \cdot \rho_V \cdot \VV_n(E)$.
\end{corollary}
\begin{proof}
	By linearity, $T$ scales uniformly all $n$-volumes of $V$,
	so we can take $E = \PP(v_1,\ldots,v_n)$
	for an $\R$-basis $(v_1,\ldots,v_n)$ of $V$,
	and $T(E) = \PP(Tv_1,\ldots,Tv_n)$.
	By \eqref{eq:outermorphism},
	$\|(Tv_1)\wedge\cdots\wedge(Tv_n)\| = |\det T| \cdot \|v_1\wedge\cdots\wedge v_n\|$, so the result follows from \Cref{pr:complex real concepts}\ref{it:blades angle}.
	\OMIT{Ou: \Cref{pr:complex real concepts}\ref{it:det angle} gives $\rho_{T(V)} \cdot \VV_n(T(E)) 
		= |\det M(Tv_1,\ldots,Tv_n)| 
		= |\det T| \cdot |\det M(v_1,\ldots,v_n)| 
		= |\det T| \cdot \rho_V \cdot \VV_n(E)$.}
\end{proof}

So, $|\det T|$ is the product of the factors by which $n$-volumes and holomorphy indices are scaled by $T$, for real $n$-subspaces $V$ with $\rho_V \neq 0$. 
By $\C$-linearity, if $V$
\SELF{Holds for any $\dim_\R V$}
is holomorphic so is $T(V)$,
and if $T$ is invertible 
\OMIT{without invertibility, $T$ can even map totally real subspaces into complex lines}
and $V$ is purely real so is $T(V)$.
We now have more details about $n$-dimensional holomorphy changes:
e.g., if $V$ is totally real then $\VV_n(T(E)) \geq |\det T| \cdot\VV_n(E)$, with equality if, and only if, $T(V)$ is totally real.
A non-invertible $T$ collapses $2n$-volumes, but we can have
$\VV_n(T(E)) \neq 0$ if $T(V)$ is not be purely real.

\begin{example}
	Let 
	$T = 
	\begin{psmallmatrix}
		1 & \im \\
		0 & 0
	\end{psmallmatrix}$ in the canonical basis $(e_1,e_2)$ of $\C^2$.
	It collapses 4-volumes,
	but 
	$T(\PP(e_1,e_2))$ can have area $\VV_2(Te_1,Te_2) = \VV_2(e_1,\im e_1) = 1 \neq 0$
	since $T(\Span_\R\{e_1,e_2\}) = \Span_\R\{e_1,\im e_1\} = (\C e_1)_\R$ is not purely real.
\end{example}

\begin{example}\label{ex:complex T 2}%
		\OMIT{\Cref{ex:complex T,ex:orientations,ex:complex T 2}, (\ref{ex:blade norms} removido)}
	In \Cref{ex:complex T},
	the image by $T$ of the unit square $\PP(e_1,e_2)$ has area $\VV_2(T e_1, T e_2) = \sqrt{G_\R(T e_1, T e_2)} =
	\begin{vsmallmatrix}
		8 & 2-2\sqrt{3} \\
		2-2\sqrt{3} & 6
	\end{vsmallmatrix}^\frac12 \cong 6.77$,
		\OMIT{$= \sqrt{32+8\sqrt{3}}$}
	so the image of the totally real $U=\Span_\R\{e_1,e_2\}$ has 
	$\rho_{T(U)} \cong \frac{4\cdot 1 \cdot 1}{6.77} \cong 0.59$, by \Cref{pr:T holomorphy}.
	For $v_1=(1-\im\sqrt{3},0)$ and $v_2=(0,1+\im)$
	we find 
	$Tv_1=(4\im,4)$, $T v_2 = (3\im-1,1-\im)$,
	$\VV_2(v_1,v_2)=2\sqrt{2}$
	and $\VV_2(T v_1, T v_2)
	= 8\sqrt{2} = |\det T|\cdot \VV_2(v_1,v_2)$,
		\OMIT{$ = 
			\begin{vsmallmatrix}
				32 & 16 \\
				16 & 12
			\end{vsmallmatrix}^\frac12$}
	and as $V=\Span_\R\{v_1,v_2\}$ is totally real so is $T(V)$ (indeed, $\im Tv_1 \perp_\R Tv_2$ and $\im Tv_2 \perp_\R Tv_1$).
\end{example}

\section{Graphical representations}\label{sc:Visual representations}

Now we discuss how to represent complex orientations, vectors and blades in ways that reflect adequately our results, and how to add blades graphically.

\subsection{Orientations.}\label{sc:Representing orientations}

In the real case, the orientation of a line/plane/space is usually indicated by a straight/curved/helix arrow, the idea being that it suggests the movement of a point/pointer/screw. 
Given any unordered basis, by ``following the movement'' one finds an order (or sign, for a single vector basis) corresponding to the chosen orientation.
	
Complex orientations can differ by more than just signs or base reorderings, so we must actually present an ordered basis with the chosen orientation.
Its order can be still be specified by some arrow, or by numbering its vectors.
For a complex line, we can also show an identification with $\C$, as in \Cref{fig:complex_orientations}.

\subsection{Vectors.}\label{sc:Representing vectors}

Any $0 \neq v \in \R^n$ carries 3 data pieces: a real line $V = \R v$ with a real orientation and a value $\|v\|$. 
As a graphical representation, the segment $\PP(v)$ is enough to determine $V$ and $\|v\|$. The orientation is indicated by an arrow at an extremity, but any ornament works to distinguish $v$ and $-v$.
	\SELF{A dot at the other extremity, to indicate the base point, would work just as well.}

Likewise, $0\neq v \in \C^n$ determines a complex line $V = \C v$ with a complex orientation and a value $\|v\|$.
But now $\PP(v)$ is not enough to determine $V = \Span_\R\{v,\im v\}$, unless one knows how the complex structure rotates $v$.
A solution is to use, instead of an arrow, a hook pointing in the direction of $\im v$, allowing one to locate $V$ and orient it via identification with $\C$ (\Cref{fig:complex vectors}).

\begin{figure}
	\centering
	\begin{subfigure}[b]{0.47\textwidth}
		\centering
		\includegraphics[width=.9\textwidth]{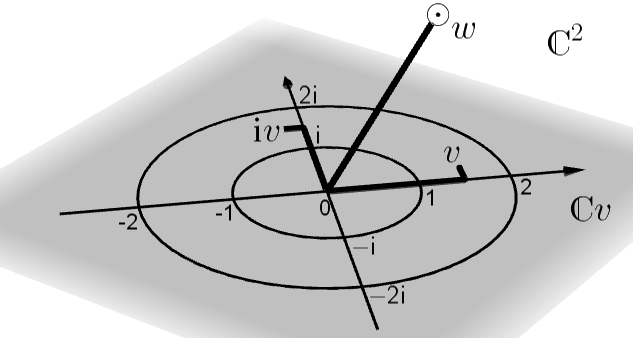}
		\caption{Complex vectors as hooked segments.}
		\label{fig:complex vectors}
	\end{subfigure} 
	\hspace{2mm}
	\begin{subfigure}[b]{0.47\textwidth}
		\centering
		\includegraphics[width=\textwidth]{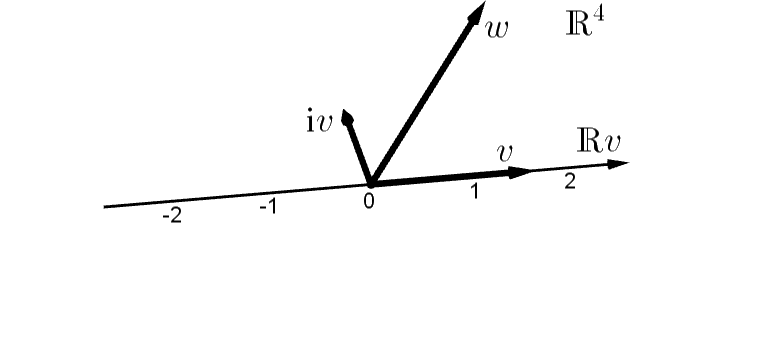}
		\caption{Real vectors as arrowed segments.}
		\label{fig:complex_vectors_real}
	\end{subfigure}
	\caption{Complex and real vectors. (a) The hook of $v$ points in the direction of $\im v$, determining $\C v \cong \Span_\R\{v,\im v\}$ with complex orientation shown via identification with $\C$. The length of $v$ is the value $\|v\|$ associated to $\C v$. Mimicking a common convention, the hook of $w$ shown as $\odot$ (\resp $\ominus$) is towards (\resp away from) a ``viewer looking down from the $4^\text{th}$ dimension'' of $\C^2 \cong \R^4$, so  the $2^\text{nd}$ imaginary component of $\im w$ is positive (\resp negative).
	(b) As a real vector in the underlying $\R^4$, $v$ determines and orients a real line $\R v$, associating the value $\|v\|$ to it.}
	\label{fig:complex x real vectors}
\end{figure}

Representing real and complex vectors differently can also help distinguish when $v$ is to be considered as a complex vector or as a real one in the  underlying real space (\Cref{fig:complex_vectors_real}).
If there is no need for this or to locate $\C v$, we can revert to representing all vectors as arrows, for simplicity.

\subsection{Blades.}\label{sc:Representing complex blades}

As seen,
	\OMIT{\Cref{sc:blades}}
only in the real case $B = v_1 \cwedge \cdots \cwedge v_p \neq 0$ can be represented (non-uniquely) by the oriented parallelotope $\PP = \PP(v_1,\ldots,v_p)$.
In the complex case, \Cref{pr:complex real concepts}\ref{it:blades angle} suggests we represent $B$ instead as a fraction%
	\footnote{As a 1-blade, any $0 \neq v \in \C^n$ is still represented by the whole $\PP(v)$, since $\rho_{\R v}=1$.}
$\rho_V$ of $\PP$, for $V=\Span_\R \{v_1,\ldots,v_p\}$ (\Cref{fig:complex blades}).
This representation reveals not only $B$, but also the holomorphy of $V$ and the volume of $\PP(v_1,\im v_1,\ldots,v_p,\im v_p)$: by \Cref{pr:blade norms}\ref{it:complex blade}, it is the volume of the Minkowski sum%
	\footnote{The Minkowski sum of $\PP(u_1,\ldots,u_k)$ and $\PP(v_1,\ldots,v_l)$ is $\PP(u_1,\ldots,u_k,v_1,\ldots,v_l)$.}
of the parallelotope fraction representing $B$ and a copy $\R$-orthogonal to it.
Note that $[B] = \Span_\C \{v_1,\ldots,v_p\} \cong V\oplus \im V$ is not the real subspace of $\PP$, 
and $B$ can also be represented by fractions (of same volume $\|B\|$) of other parallelotopes in $[B]$ (even outside $V$), as long as they admit the same complex orientation.

\Cref{pr:T holomorphy} now gives another interpretation for $|\det T|$, as the scaling factor of parallelotope fractions representing complex $n$-blades.

\begin{example}\label{ex:figura6}%
	Let $B = u\wedge v = u \wedge v' = u'' \wedge v''$, with $u$ and $v$ as in \Cref{ex:figura3}, $v' = 3 e_2$, $u'' = 3 e_1$ and $v'' = (3+6\im) e_1 + 2 e_2$.
	\Cref{fig:complex blades} represents $B$ as different fractions (of same area $\|B\|=6$) of the parallelograms of real blades $u\rwedge v \neq u \rwedge v' \neq u'' \rwedge v''$.
	The complex vector hooks 
	remind us that $[B] = \Span_\C\{u,v\} \cong \Span_\R\{u,\im u,v,\im v\}$ is not the real plane of the parallelogram.
	Curved arrows show the order of the basis $(u,v)$ giving the complex orientation of $[B]$.	
	Not every parallelogram fraction of area $6$ in $[B]$ represents $B$, due to complex orientations: e.g., \Cref{fig:exemplo2} represents $\im B$, not $B$.
	The parallelograms $\PP$ of the real blades have different areas, but each forms a different angle with $\im \PP$ (not shown), so that each $\PP(u,\im u,v,\im v)$ has 4-volume $\|B\|^2=36$.

	In \Cref{ex:totally real}, $V = \Span_\R\{u,v\}$ is totally real, so $\PP(u,v)$ can represent both $u \rwedge v = 20 f_{24}$ and $u \cwedge v  = -20 e_{12}$,
	but $[u \rwedge v] = V  = \Span_\R\{f_2,f_4\}$ with real orientation of $f_{24}$,
	and $[u \cwedge v] = \Span_\C \{u,v\} = \C^2$ with complex orientation opposite that of $e_{12}$.
	In \Cref{ex:square2}, $\Span_\R\{u,v\}$ is holomorphic, so $u \wedge v = 0$ is a zero fraction of $S$.
	In \Cref{ex:C3 nova,ex:3 em C3 parte1}, $u \wedge v \wedge w$ and $u \wedge v$ correspond, respectively, to $81.6\%$ of $\PP(u,v,w)$ and $29.4\%$ of $\PP(u,v)$, by \Cref{tab:exemplos}.
\end{example}

\begin{figure}
	\centering
	\begin{subfigure}[b]{0.27\textwidth}
		\centering
		\includegraphics[width=.85\textwidth]{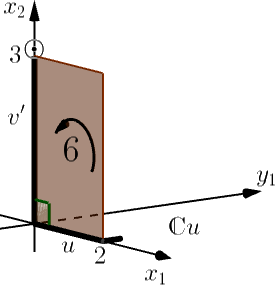}
		\caption{$\Span_\R\{u,v'\}$ is totally real, so $B$ is shown as the whole $\PP(u,v')$.}
		\label{fig:exemplo4}
	\end{subfigure} 
	\hspace{.05mm}
	\begin{subfigure}[b]{0.27\textwidth}
		\centering
		\includegraphics[width=.9\textwidth]{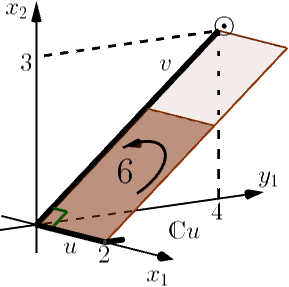}
		\caption{$V=\Span_\R\{u,v\}$ has $\rho_V = 0.6$, so $B$ is shown as 60\% of $\PP(u,v)$.}
		\label{fig:exemplo5}
	\end{subfigure} 
	\hspace{.05mm}
	\begin{subfigure}[b]{0.42\textwidth}
		\centering
		\includegraphics[width=\textwidth]{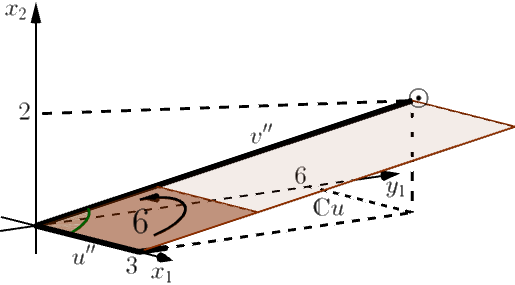}
		\caption{$\Span_\R\{u'',v''\}$ is closer to the complex line $\C u$, so $B$ is an even smaller fraction of $\PP(u'',v'')$.}
		\label{fig:exemplo6}
	\end{subfigure}
	\caption{Different representations of the same complex blade $B$, of \Cref{ex:figura6}, as oriented fractions of parallelograms (shown darker than the whole parallelograms).}
	\label{fig:complex blades}
\end{figure}

\subsection{Graphical sums.}

Blades $A,B \in \bigwedge^2 \F^n$ with $[A] \cap [B] \neq \{0\}$ decompose as $A = a \wedge c$ and $B= b \wedge c$ with $c \in [A]\cap[B]$ and $a,b \in \F^n$, which can be chosen $\F$-orthogonal to $c$.
In the real case, $\PP(a,c)$ and $\PP(b,c)$ represent $A$ and $B$, and a representation of $A+B = (a+b)\wedge c$ is  obtained adding $a + b$ graphically and then taking $\PP(a+b,c)$.

In the complex case, if $a$ and $b$ are $\C$-orthogonal to $c$ then $\Span_\R\{a,c\}$, $\Span_\R\{b,c\}$ and $\Span_\R\{a+b,c\}$ are totally real,
so $A$, $B$ and $A+B$ are again represented by $\PP(a,c)$, $\PP(b,c)$ and $\PP(a+b,c)$.
But if not, then $A$ and $B$ correspond to fractions whose graphical sum does not represent $A+B$,
and one must add graphically $\PP(a,c)$ and $\PP(b,c)$ to find $\PP(a+b,c)$, compute the holomorphy index of $\Span_\R\{a+b,c\}$, and then take the corresponding fraction.

The idea extends to blades $A,B \in \bigwedge^p \C^n$ with $\dim_\C ([A] \cap [B]) = p-1$,
decomposed as $A = a \wedge C$ and $B= b \wedge C$ with $C \in \bigwedge^{p-1}([A]\cap [B])$, which can always be represented by a parallelotope in a totally real subspace of $([A]\cap[B])_\R$.

\begin{example}\label{ex:sum}
	A graphical subtraction of the darker parallelograms in \Cref{fig:exemplo4,fig:exemplo5}, both representing $B$, would erroneously give a nonzero result.
	Subtracting the parallelograms of $u \rwedge v$ and $u\rwedge v'$ gives another representing $u \rwedge (v-v') = u \rwedge 2\im u \neq 0$  in $\C u$,
	and a fraction $\rho_{\C u}=0$ of it represents, correctly, $B-B = u \cwedge 2\im u = 0$.
 
	The whole parallelogram of \Cref{fig:exemplo2} represents  $\im B = \im u \wedge v$, since $\Span_\R\{\im u,v\}$ is totally real.
	The graphical sum of $u \rwedge v$ and $(\im u)\rwedge v$ gives the larger parallelogram in \Cref{fig:sum}, of area $\|(u+\im u)\rwedge v\| = \sqrt{136}$.
		\OMIT{$v$ projects on $\im u$, not $u+\im u$}
	As $B+\im B = (1+\im) u \wedge v$ has $\|B+\im B\| =  6\sqrt{2}$, it is represented by $\rho = \frac{6\sqrt{2}}{\sqrt{136}} \cong 73\%$ of this parallelogram.
	In \Cref{fig:exemplo4}, $v'$ is $\C$-orthogonal to $u$, so $B= u \wedge v'$ and $\im B = \im u \wedge v'$ can also be represented by $\PP(u,v')$ and $\PP(\im u,v')$, and $B+\im B$ by the whole $\PP(u+\im u,v')$, as in \Cref{fig:sum2}.
\end{example}

\begin{figure}
	\centering
	\begin{subfigure}[b]{0.45\textwidth}
		\centering
		\includegraphics[width=.72\textwidth]{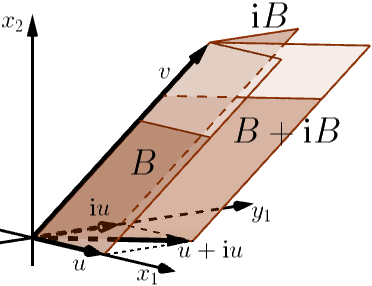}
		\caption{Parallelograms in purely real planes}
		\label{fig:sum}
	\end{subfigure} 
	\hspace{.5mm}
	\begin{subfigure}[b]{0.45\textwidth}
		\centering
		\includegraphics[width=.5\textwidth]{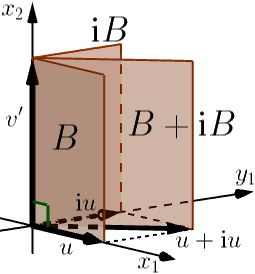}
		\caption{Parallelograms in totally real planes}
		\label{fig:sum2}
	\end{subfigure}
	\caption{Graphical sums giving $B+\im B$, from \Cref{ex:sum}.}
	\label{fig:sums}
\end{figure}

\section{Clifford algebras and Hermitian geometry}\label{sc:Complex numbers and Geometric Algebra}

Our results facilitate applying complex Grassmann algebras to Hermitian geometry.
Contractions (interior products) and regressive products also work well \cite{Mandolesi_Contractions,Mandolesi_Contractions2}.
We now discuss whether the whole apparatus of Clifford algebras can be used to study general Hermitian spaces.

Specifically, we would like to know if this can be done in a practical way, and preserving the isotropy of such spaces: no canonical isomorphism with $\C^n$, no natural concept of complex conjugation, and no distinction between real or imaginary vectors.
This is important for some applications:
e.g., in quantum theory, vectors differing by a complex phase represent the same physical state, so there are no distinctly real or imaginary vectors;
in gauge theory, complex gauge transformations might not preserve real structures (decompositions of complex spaces into real and imaginary parts, or, equivalently, antilinear involutions working as complex conjugation).

This requirement excludes Hermitian Clifford Analysis \cite{Brackx2008,Sabadini2014}, which uses complex conjugation to define a Hermitian product on the Clifford algebra of $\C^n$,
and Unitary Geometric Algebra \cite{Sobczyk1993,Sobczyk2012unitary}, whose complex vectors $x+\im y$, for $x,y \in \R^n$ and a real Clifford algebra element $\im$ of square $-1$,
have discernible real and imaginary parts ($x$ and $\im y$ have different grades).
	\CITE{Sobczyk2013 does the same; Sobczyk1993 uses a bivector as $\im$; Sobczyk2001universal uses complex algebra} 
Other methods \cite{Hrdina2022,Marchuk2008,Sobczyk2001universal,Stoica2018} have similar problems.

The geometry of Euclidean spaces is easily linked to real Clifford algebras because the real inner product and the scalar part of the Clifford one can be made equal, $\inner{u,v} = (uv)_0$.
This fails in the Hermitian case, as the inner product is sesquilinear and the Clifford one is $\C$-bilinear.
Putting both on a complex space would require a compatibility condition to ensure they describe the same geometry,
and though $\Real\inner{u,v} = \Real(uv)_0$ might seem natural, it does not work, as $\Real\inner{\im u,\im v} = \Real\inner{u,v}$ while
$\Real\left((\im u)(\im v)\right)_0 = - \Real(uv)_0$.
In fact, any complex Clifford product breaks the isotropy: we can identify a vector $v$ as real if $v^2 > 0$, or imaginary if $v^2<0$.

Geometric algebraists might say that, instead of Hermitian geometry, one should use the geometry of the underlying real space.
As Hestenes recast real Clifford algebras as Geometric Algebra (GA), he rejected using $\C$ as a scalar field (\hspace{1sp}\cite[pp.\,9--11]{Hestenes1986unified}, \cite[pp.\,xii--xiii]{Hestenes1984}), seeing it as just a subalgebra of GA, like $\mathit{Cl}^+(\R^2)$.
He hoped that GA as a universal language would break barriers between various areas of Mathematics and Physics, and free students from learning different formalisms like complex numbers, quaternions, matrices, etc.
Lack of a geometric interpretation of complex blades may have also played a role in his view, still shared by many and reinforced by the fact that Geometric Calculus extends Complex Analysis, and complex elements of Quantum Theory, like spinors, were reformulated using GA.

But a specialized tool or formalism, that exploits special features of a problem, can at times be preferable to an all-purpose one.
And, as the algebraic closure of $\R$, it is natural for $\C$ to emerge in real formalisms:
e.g., in the invariant decomposition of a real bivector \cite{Roelfs2023}.
Even Hestenes admits complex roots in \cite[p.\,74]{Hestenes1984}, albeit noting that they could be identified with multivectors.
But possible does not necessarily mean convenient.

Rewriting a complex formalism in real terms is not always an advantage, as it can become untidy (e.g., imagine a widespread use of 
$\begin{psmallmatrix}
	x & -y \\
	y & x
\end{psmallmatrix}$ instead of $x+\im y$).
Hermitian geometry can indeed be done in the underlying real space with a complex structure, but it loses elegance and efficiency, even with the use of GA.
The dimension of the space doubles, squaring that of its Clifford or Grassmann algebras, as in the examples of \Cref{sc:blades}.
In attempts to express complex concepts via GA,
the language quickly becomes cumbersome
(e.g., a ``complex $p$-subspace of $\C^n\,$'' becomes a ``real $2p$-subspace of $\R^{2n}$ that is closed under an action of $\mathit{Cl}^+(\R^2)$''),
and to simplify it one ends up reverting to something akin to the complex formalism, only with different words.

Hopefully, our results can be a step towards understanding if, when or how Clifford algebras can be used to study complex geometry, or perhaps they might point the way towards an alternative formalism.

\section{Conclusion}\label{sc:Conclusion}

Here we summarize our main results and discuss possible developments.

\subsection{Main results}

The norm of a complex blade $B = v_1 \wedge \cdots \wedge v_p$, for $v_1,\ldots,v_p \in \C^n$, has been related to volumes in two different ways:
\begin{enumerate}[label = (\roman*)]
	\item it is a fraction $\rho_V$ of the $p$-volume of the parallelotope $\PP = \PP(v_1,\cdots,v_p)$, where $\rho_V$ is the holomorphy index of $V=\Span_\R\{v_1,\ldots,v_p\}$;\label{interpr i}
	\item its square is the $2p$-volume of $\hat{\PP} = \PP(v_1,\im v_1,\cdots,v_p, \im v_p)$, the parallelotope spanned by those vectors and their rotations by $\im$.\label{interpr ii}
\end{enumerate}
When $p=n$, \ref{interpr i} and \ref{interpr ii} also hold for $|\det M|$, where $M = M(v_1,\ldots,v_n)$.

\ref{interpr i} is our main result, allowing us to represent complex blades like real ones, except that $B$ corresponds to only a fraction of $\PP$, smaller the closer $V$ is to becoming or containing a holomorphic subspace.

\ref{interpr ii} may seem less natural, but was easier to prove. By definition, $\|B\|^2$ is a Gram determinant, which equals $|\det M|^2$ if $p=n$ (we can assume it by reducing the ambient space to $V$).
And for the determinant we gave many proofs, some direct but abstract, others more geometric: e.g., using shears to turn $\hat{\PP}$ into a box and diagonalize $M$ (up to a reordering of columns) while preserving the volume and $|\det M|$.

\ref{interpr i} and \ref{interpr ii} are connected by the fact that $\hat{\PP}$ is the Minkowski sum of $\PP \subset V$ and its rotated copy $\im \PP \subset \im V$, so its $2p$-volume is the product of the equal $p$-volumes of $\PP$ and $\im \PP$ times the sine of the disjointness angle of $V$ and $\im V$. The role of this sine, which equals $\rho_V^2$, is to project the volume of $\im \PP$ on the orthogonal complement of $V$.

The index $\rho_V$ measures the lack of holomorphy of $V$ in terms of how much the complex structure displaces it: a holomorphic $V$ is not displaced, rotating inside itself, so $\rho_V = 0$; the maximum $\rho_V = 1$ occurs for a totally real $V$, which rotates $\R$-orthogonally \wrt itself.
One can compute $\rho_V$ via the disjointness angle, the Kähler angles of $V$ (the principal angles of $V$ and $\im V$, minus the inherent duplicates), or as a ratio of norms of complex and real blades formed by $v_1,\ldots,v_p$.

The argument of the complex determinant of a matrix $M \in GL(\C^n)$ has been interpreted as a phase difference, the change of complex orientation caused by $M$.
There are various equivalent ways to consider the complex orientation of a complex $p$-subspace $V$, most notably as:
\begin{itemize}
	\item an element in the unit circle of $\bigwedge^p V \cong \C$ (so, a unit $p$-blade $B$ in $V$);
	
	\item an isometric isomorphism between $\bigwedge^p V$ and $\C$, with the ray through $B$ mapped to the positive real axis;
	
	\item an equivalence class of (orthonormal) ordered bases of $V$ differing by transformations of positive determinant, with $B$ being the normalized blade formed by the basis vectors.
\end{itemize} 
The phase difference $\varphi$ between two orientations is the angle between the circle elements or rays,
and the argument of the determinant of a transformation between bases in their classes.
The transformation can be decomposed into elementary ones, and only phase rotations%
\footnote{Complex transpositions or generalized complex line reflections are phase rotations by $\pi$.}
affect orientations, 
so that $\varphi$ is just the sum (mod $2\pi$) of their angles.

\subsection{Possible developments}

Our results may open new ways to apply complex Grassmann algebras to complex geometries (Hermitian, Kähler, algebraic, etc.).
Traditionally, complex spaces are seen as naturally oriented (a real orientation given by the complex structure), and forms used on them are based on a complexification of the exterior algebra of the underlying real space \cite{Goldman1999,Griffiths1994}.
This approach has been very successful, but complex orientations and forms based directly on the complex exterior algebra might, in light of our results, be also useful for some problems. 
Promising signs are the complex volumetric Pythagorean theorems \cite{Mandolesi_Pythagorean}, 
and Hodge-like star operators obtained via complex orientations in \cite{Mandolesi_Contractions}, arguably simpler and more well suited for complex geometry than the usual complex Hodge star operator (which acts on complexified forms in the underlying real space).

The geometric interpretation of complex blades in \Cref{pr:blade norms}\ref{it:complex blade} has proven useful \cite{Mandolesi_Pythagorean,Mandolesi_Born,Mandolesi_Grassmann,Mandolesi_Products,Mandolesi_TotalGrassmannian,Mandolesi_Contractions},
and we expect the simpler \Cref{pr:complex real concepts}\ref{it:blades angle} to be even more fruitful, as it relates blade norms to volumes and also the holomorphy index.
This index, a surprising application of the disjointness angle, might help extend to higher dimensional subspaces results involving Kähler angles of real planes \cite{Goldman1999,Rosenfeld1997,Tasaki2001}.
It is useful that the index can be obtained via \Cref{pr:blades complex real}\ref{it:blade complex real}, with no need to compute principal angles.

Complex blades carry information about volumes, complex orientations and subspaces, and as our aim was to extract these pieces of information, we kept them separated.
But for some purposes it can be convenient to combine volume and complex orientation into a complex volume, just as signed volumes are sometimes used in the real case \cite{Cameron2019}.
Indeed, in \cite[Thm.\,5.1]{Mandolesi_Identities} we obtained a nice trigonometric identity for asymmetric angles (which measure volume contraction) by appropriately incorporating (real or complex) orientations of subspaces into them.

As discussed, Clifford algebras seem incompatible with Hermitian spaces, under certain conditions.
The better understanding of how complex Grassmann algebras relate to Hermitian geometry, provided by our results, 
might perhaps show the way towards some generalization or alternative to Clifford algebras, with a new product compatible with the Hermitian one.

\appendix

\section{Angles between subspaces}\label{Angles between subspaces}

In high dimensions, there are different concepts of angle between subspaces, some of which we review here.

Let $V,W\subset \F^n$, $p=\dim V$, $q=\dim W$ and $m=\min\{p,q\}$.
Their separation is described by \emph{principal angles} \cite{Bjorck1973,Galantai2006} $0 \leq \theta_1\leq\cdots\leq \theta_m \leq \frac\pi2$
and associated \emph{principal bases} $(v_1,\ldots,v_p)$ and $(w_1,\ldots,w_q)$, which are orthonormal bases with $\inner{v_j,w_k} = 0$ for $j \neq k$ and $\inner{v_j,w_j} = \cos \theta_j$ for $1\leq j\leq m$.
They can be defined recursively: 
for $1\leq j\leq m$, 
$V_j = \{v\in V: \inner{v,v_k} = 0 \ \forall\, k<j\}$
and
$W_j = \{w\in W: \inner{w,w_k} = 0 \ \forall\, k<j\}$,
let
\begin{equation}\label{eq:theta_j}
	\theta_j = \inf\{\theta_{v,w} : 0\neq v \in V_j,\  0 \neq w \in W_j\},
\end{equation}
and choose as its \emph{principal vectors} unit $v_j \in V_j$ and $w_j \in W_j$
with $\theta_{v_j,w_j} = \theta_j$.
Complete the larger orthonormal basis as needed.
Note that $P_V w_j = v_j \cos\theta_j$, where $P_V$ is the orthogonal projection on $V$.

Principal angles are often combined into various concepts of distance between $V$ and $W$. 
If $p=q$, their \emph{Fubini-Study distance} is the angle
$\theta_{\mathrm{FS}} = \cos^{-1}\big(\prod_{j=1}^p \cos\theta_j\big)$.
It extends for $p \neq q$ as an asymmetric angle%
\footnote{Previously called \emph{Grassmann angle} in \cite{Mandolesi_Grassmann,Mandolesi_Identities}.}
$\Theta_{V,W}$ \cite{Mandolesi_Products,Mandolesi_TotalGrassmannian},
used to define another angle $\Upsilon_{V,W}$, more convenient for our purposes:

\begin{definition}\label{df:angles}
	The \emph{asymmetric angle} of $V$ with $W$ is $\Theta_{V,W} = \cos^{-1} \frac{\|P_W B\|}{\|B\|}$,
	where $B$ is a blade with $[B]=V$ and $P_W$ is the orthogonal projection%
	\footnote{If $B=v_1\wedge \cdots \wedge v_p$ then $P_W B = (P_W v_1)\wedge \cdots \wedge (P_W v_p)$.}
	on $W$.
	Their \emph{disjointness angle} is $\Upsilon_{V,W}=\frac{\pi}{2}-\Theta_{V,W^\perp}= \sin^{-1} \frac{\|P_{W^\perp} B\|}{\|B\|}$.
\end{definition}

By \Cref{pr:blade norms}, $\cos \Theta_{V,W}$ (squared, if $\F=\C$) measures the contraction of $p$-volumes ($2p$, if $\F=\C$) orthogonally projected from $V$ to $W$.
Thus we have $\Theta_{V,W} = 0 \Leftrightarrow V \subset W$, and $\Theta_{V,W} = \frac\pi2 \Leftrightarrow W^\perp \cap V \neq \{0\}$.
In \cite{Mandolesi_Products},
we show that $\Theta_{V,W} = \cos^{-1}\big(\prod_{j=1}^p \cos\theta_j\big)$ if $p \leq q$, and $\Theta_{V,W} = \frac\pi2$ if $p>q$, 
so generally $\Theta_{V,W} \neq \Theta_{W,V}$ for $p \neq q$,
reflecting natural asymmetries between the subspaces.

\begin{proposition}\label{pr:upsilon}
	Let $V$, $W$ and $\theta_j$ be as above, and $A,B \in \bigwedge \F^n$ be blades.
	\begin{enumerate}[label = (\roman*)]
		\item $\sin \Upsilon_{V,W}$ (squared, if $\F=\C$) is the contraction factor of $p$-volumes ($2p$, if $\F=\C$) orthogonally projected from $V$ to $W^\perp$. \label{it:projection Upsilon}
		
		\item $\Upsilon_{V,W}=0 \Leftrightarrow V\cap W\neq\{0\}$, and $\Upsilon_{V,W}=\frac\pi2 \Leftrightarrow V\perp W$.\label{it:Upsilon 0 pi2}
		
		\item $\Upsilon_{V,W} = \Upsilon_{W,V} = \sin^{-1} \big(\prod_{j=1}^m \sin\theta_j\big)$. \label{it:Upsilon sin}
		
		\item $\|A \wedge B\| = \|A\|\|B\| \sin \Upsilon_{[A],[B]}$. \label{it:norm wedge Upsilon}
	\end{enumerate}
\end{proposition}
\begin{proof}
	Follows from the above properties of $\Theta_{V,W}$,
	and in \cite{Mandolesi_Products} we show 
	$\Theta_{V,W^\perp} = \Theta_{W,V^\perp} = \cos^{-1}\big(\prod_{j=1}^m \sin\theta_j\big)$
	and $\|A \wedge B\| = \|A\|\|B\| \cos\Theta_{[A],[B]^\perp}$.
\end{proof}

\begin{figure}
	\centering
	\includegraphics[width=0.5\linewidth]{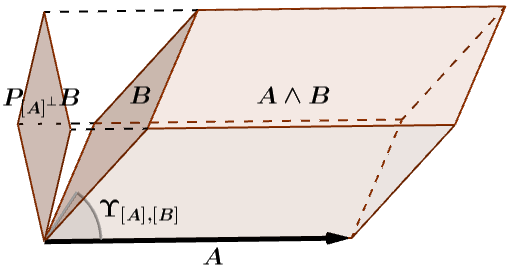}
	\caption{$\sin \Upsilon_{[A],[B]} = \frac{\|P_{[A]^\perp} B\|}{\|B\|}$ is the ratio between the areas of the parallelograms. The volume of the parallelepiped is $\|A\wedge B\| = \| A\|\|P_{[A]^\perp} B\| = \|A\|\|B\|\sin \Upsilon_{[A],[B]}$.}
	\label{fig:angulo}
\end{figure}

See \Cref{fig:angulo}.
While $\Theta_{V,W}$ measures how far $V$ is from being contained in $W$, \ref{it:Upsilon 0 pi2} shows $\Upsilon_{V,W}$ measures how far $V$ and $W$ are from intersecting non-trivially
(how disjoint they are, hence the name).
By \ref{it:Upsilon sin}, $\Upsilon_{V,W}$ tends to be smaller than the smallest principal angle, being quite small if $\theta_j < \frac\pi2$ for many $j$.
This makes sense since, by \ref{it:projection Upsilon}, it reflects volume contraction resulting from various length contractions by the $\sin \theta_j$'s.
In \ref{it:norm wedge Upsilon}, which  generalizes \eqref{eq:norm 2-blades}, $\wedge$ is the exterior product of $\bigwedge \F^n$, so in $\bigwedge (\C^n)_\R$ it must be the $\R$-bilinear $\rwedge$.

Note that $\Upsilon_{V,W}$ is actually an angle in $\bigwedge^p\F^n$, 
between the line $\bigwedge^p V$ and the subspace $(\bigwedge^p(W^\perp))^\perp$.
In the complex case,  $\Upsilon_{V,W} \neq \Upsilon_{V_\R,W_\R}$, 
which can be understood since $\bigwedge^p \C^n \neq \bigwedge^p (\C^n)_\R$,
or via \ref{it:Upsilon sin}, as $V_\R$ and $W_\R$ have the same principal angles as $V$ and $W$, but twice repeated.

\vspace{6pt}

\noindent
\textbf{Note.} This article has been posted to the arXiv e-print repository, with the identifier arXiv:2403.17022

\vspace{6pt}

\noindent
\textbf{Data Availability.} Not applicable.


\end{document}